\documentclass[reqno,11pt]{amsart}
\usepackage{amsmath,amsthm,amssymb,amsfonts,color,pifont,cite}
\usepackage[normalem]{ulem}
\usepackage{enumerate,hyperref}
\usepackage[normalem]{ulem}

\newcounter{corr}
\definecolor{violet}{rgb}{0.580,0.,0.827}
\newcommand{\corr}[3]{\typeout{Warning : a correction remains in page
		\thepage}
	\stepcounter{corr}        
	{\color{blue}\ifmmode\text{\,\sout{\ensuremath{#1}}\,}\else\sout{#1}\fi}
	{\color{red}#2}
	{\color{violet} #3}}

\usepackage{pgf,tikz}
\usepackage{mathrsfs,ulem}
\usetikzlibrary{arrows}

\newtheorem{theorem}{Theorem}
\newtheorem{proposition}[theorem]{Proposition}
\newtheorem{definition}[theorem]{Definition}
\newtheorem{lemma}[theorem]{Lemma}
\newtheorem{corollary}[theorem]{Corollary}
\newtheorem{remark}[theorem]{Remark}

\newcommand{\R}{\mathbb R}

\newcommand{\A}{\mathcal A}
\newcommand{\K}{\mathcal K}
\renewcommand{\S}{\mathcal S}
\newcommand{\oE}{\overline{E}_\lambda}
\newcommand{\Om}{\Omega}
\newcommand{\eps}{\varepsilon} \def\R{\mathbb R} \def\K{\mathcal K}
\def\I{\mathcal I} 
\def\E{\mathcal E}

\usepackage{a4wide}

\author{Cyrill B. Muratov} \address{Department of Mathematical
  Sciences, New Jersey Institute of Technology, Newark, NJ 07102, USA}
\email{muratov@njit.edu}

\author{Matteo Novaga} \address{Department of Mathematics, University
  of Pisa, Largo B. Pontecorvo 5, 56127 Pisa, Italy}
\email{matteo.novaga@unipi.it}

\author{Berardo Ruffini} \address{Institut Montpelli\'erain Alexander
  Grothendieck, Universit\'e de Montpellier, place Eugene Bataillon
  34095, Montpellier Cedex 5, France, and Dipartimento di Matematica,
  Alma Mater Studiorum -- Universit\`a di Bologna, Piazza di Porta San
  Donato 5, 40126 Bologna, Italy} \email{berardo.ruffini@unibo.it}

\numberwithin{equation}{section}

\title[Conducting flat drops]{Conducting flat drops in a
  confining potential}

\begin{document}

\begin{abstract}
  We study a geometric variational problem arising from modeling
  two-dimensional charged drops of a perfectly conducting liquid in
  the presence of an external potential. We characterize the
  semicontinuous envelope of the energy in terms of a parameter
  measuring the relative strength of the Coulomb interaction. As a
  consequence, when the potential is confining and the Coulomb
  repulsion strength is below a critical value, we show existence and
  partial regularity of volume-constrained minimizers. We also derive
  the Euler--Lagrange equation satisfied by regular critical points,
  expressing the first variation of the Coulombic energy in terms of
  the normal $\frac12$-derivative of the capacitary potential.
\end{abstract}

\maketitle



\tableofcontents 

\section{Introduction}

This paper is concerned with a geometric variational problem modeling
charged liquid drops in two space dimensions, whose study was
initiated in \cite{MurNovRuf}.  The problem in question arises in the
studies of electrified liquids, and one of its main features is that
the Coulombic repulsion of charges competes with the the cohesive
action of surface tension and tends to destabilize the liquid drop
\cite{Tay,fernandezdelamora07,Ray}, an effect that is used in many
concrete applications (see, e.g.,
\cite{gaskell97,barrero07,castrohernandez15}).  From a mathematical
point of view, this problem is interesting due to the competition
between short-range attractive and long-range repulsive forces that
produces non-trivial energy minimizing configurations and even
nonexistence of minimizers when the total charge is large enough {(for
  an overview, see \cite{cmt:nams17})}.  The original model in three
dimensions was proposed by Lord Rayleigh \cite{Ray} and later
investigated by many authors (see, for example, \cite{Tay,FonFri,
  basaran89a,burton11,garzon14,gnrI,gnrII,MurNov,DHV}; this list is
not meant to be exhaustive).

In mathematical terminology, we are interested in
the properties of the energy 
\begin{equation}\label{problemintro}
  E_\lambda(\Om) := \mathcal H^1(\partial \Omega) + \lambda\,
  \I_1(\Om) + \int_{\Om} g(x)\,dx, 
\end{equation}
where $\Omega\subset\R^2$ is a compact set with smooth boundary and
prescribed area $|\Omega| = m$, 
\begin{equation}
  \label{eqIk}
  \I_1(\Omega):=\inf_{\mu \in \mathcal P(\Omega)} 
  \int_\Omega \int_\Omega \frac{d\mu(x)\,d\mu(y)}{|x-y|},
\end{equation}
where $\mathcal P(\Omega)$ is the space of probability measures
supported on $\Omega$, and $g$ is a continuous function.  The function
$\mathcal I_1(\Omega)$ is often referred to as the 1-Riesz capacitary
energy of $\Omega$, and the right-hand side of \eqref{eqIk} admits a
unique minimizer $\mu_\Om$, which is called the {\em
  equilibrium measure} of $\Om$ \cite{landkof}.  Physically, the terms
in {\eqref{problemintro}} are, in the order of appearence: the excess
surface energy of a flat drop, the self-interaction Coulombic energy
of a perfect conductor carrying a fixed charge, and the effect of an
external potential.  The first term in {\eqref{problemintro}} acts as
a cohesive term. In contrast, the second term is a capacitary term due
to the presence of a charge and acts on the drop as a {repulsive}
term. The parameter $\lambda > 0$ measures the relative strength of
Coulombic repulsion. We refer to \cite{MurNovRuf} for a more
comprehensive derivation of the two-dimensional model, as well as for
a deeper physical background.

A minimization problem for {\eqref{problemintro}} must take into
account the fine balance that exists between the surface and the
capacitary term \cite{MurNovRuf}. A rough prediction of the behavior
of minimizers, when they exist, is that if $\lambda$ is big enough,
then the drop will tend to be unstable, possibly leading to absence of
a minimizer at all, while if $\lambda$ is small, the {dominant} term
is the surface one, leading to existence and stability of {energy
  minimizing} drops in a suitable class of sets. One of the purposes
of this paper is to make the above prediction precise.

We point out that the energy above is a particular case of the more
general energy
\begin{equation}\label{genmin}
  E_{\lambda,\alpha,N}(\Omega) :=  \mathcal H^{N-1}(\partial
  \Omega)+\lambda 
  \I_\alpha(\Omega)+\int_{\Omega }g(x) \, dx, \qquad 
  \mathcal I_\alpha(\Omega)
  = \inf_{\mu \in \mathcal P(\Omega)} 
  \int_\Omega \int_\Omega \frac{d\mu(x)\,d\mu(y)}{|x-y|^\alpha},
\end{equation}
where $\Omega \subset \mathbb R^N$ is a compact set with smooth
boundary and with prescribed Lebesgue measure $|\Omega| = m$, and
$\alpha \in (0, N)$.  For $g \equiv 0$, some mathematical analysis of
the minimization problem associated with \eqref{genmin} has been
carried out in \cite{gnrI,gnrII}, where it was shown that the problem
is ill-posed for $\alpha <N- 1$. Indeed, the nonlocal term
$\I_\alpha(\Om)$ is finite whenever the Hausdorff {dimension} of {a
  compact set} $\Om$ is {greater} than $N-\alpha$. On the other hand,
the Hausdorff measure $\mathcal H^{N-1}$ is {trivially} null on sets
whose Hausdorff {dimension} is less than $N-1$. Thus, whenever a
positive gap between $N-\alpha$ and $N-1$ occurs, it is possible to
construct sets with $\I_\alpha$ positive and finite, but zero
$\mathcal H^{N-1}$-measure, ensuing non-existence of minimizers
\cite{gnrI}. The existence of a minimizer for \eqref{genmin} in the
case $g = 0$ and $\alpha \geq N-1$ is still open, except in the
borderline case $N = 2$ and $\alpha = 1$ (see \cite{MurNovRuf}). In
this latter case, we showed that there exists an explicit threshold
{$\lambda = \lambda_c(m)$, where
  \begin{align}
    \label{eq:lamc}
    \lambda_c(m): = {4 m\over \pi},    
  \end{align}
}such that for $\lambda>\lambda_c(m)$ no minimizer {exists}, while for
$\lambda\le\lambda_c(m)$ the only minimizer is a ball of measure $m$.

In this paper we address the question of existence and qualitative
properties of minimizers of \eqref{genmin} for $N = 2$ and
$\alpha = 1$.  To that aim, in Theorem \ref{teorel} we characterize
the lower semicontinuous envelope of the energy $E_\lambda$ with
respect to the $L^1$ topology.  As a corollary, we show that the
energy $E_\lambda$ is lower semicontinuous as long as $\lambda$ is
below the precise {threshold} $\lambda_c(m)${, which is the same as
  the one for the case $g \equiv 0$}.  Then in Theorem
\ref{teoexistence} we prove, under a suitable coercivity assumption on
$g$, the existence of volume-constrained minimizers for $E_\lambda$,
as long as $\lambda\le \lambda_c(m)$.  Furthermore, in Theorem
\ref{teodensity} we obtain density estimates and finiteness of the
number of connected components of minimizers.  Building on these
partial regularity results, in Theorem \ref{teoasy} we consider the
asymptotic regime $\lambda,\,m\to 0$, with
$\limsup (\lambda/\lambda_c(m)) < 1$.  In this limit the potential
term is of lower order, and we show that minimizers, suitably
rescaled, tend to a ball, which is the unique minimizer when $g=0$
\cite{MurNovRuf}.  Finally, in Theorem \ref{teocinque} we compute the
first variation of $\I_1$ and, as a consequence, we derive the
Euler--Lagrange equation of the functional $E_\lambda$, for
sufficiently smooth sets.

\smallskip

\paragraph{\bf Acknowledgements} The work of CBM was partially
supported by NSF via grants DMS-1614948 and DMS-1908709. MN has been
supported by GNAMPA-INdAM and by the University of Pisa via grant PRA
2017-18. BR was partially supported by the project ANR-18-CE40-0013
SHAPO financed by the French Agence Nationale de la Recherche (ANR)
and the GNAMPA-INdAM Project 2019 "Ottimizzazione spettrale non
lineare".

\section{Statement of the main results}

As was mentioned earlier, the main difficulty in showing existence of
minimizers for the variational problems above is that adding a surface
term to a nonlocal capacitary term typically leads to an ill-posed
problem. The strategy adopted in \cite{MurNovRuf} to study the
  minimizers of \eqref{problemintro} with $g = 0$ was to show
directly a lower bound on the energy, given by that of a single
ball, using some concentration compactness tools and some fine
properties of the theory of convex bodies in dimension two. The
presence of the bulk energy in \eqref{problemintro} precludes
application of these techniques.

The strategy of this paper is different: We first characterize the
lower semicontinuous envelope of the functional $E_\lambda$, in a
class of sets which includes compact sets with smooth boundary.  Its
explicit expression allows us to state that for certain values of
$\lambda$, the energy $E_\lambda$ is lower semicontinuous with respect
to the $L^1$ convergence. To state the main results of the paper, we
introduce some notation. Given $m>0$, we denote by $\A_m$ the class of
all measurable subsets of $\R^2$ of measure $m$:
\begin{align}
  \label{Am}
  \A_m := \left\{ \Omega\subset\R^2\,: \, |\Omega| = m \right\}. 
\end{align}
We then introduce the families of sets
\begin{align}\noindent
 \S_m&:=\left\{\Omega\in\A_m\,: \,\Omega \text{ compact},
        \, \partial\Omega \text{ smooth} \right\}, 
\\  \label{Km}
  \K_m&:=\left\{\Omega\in\A_m\,: \,\Omega \text{ compact},
        \, \mathcal H^1(\partial \Omega)<+\infty \right\}.
\end{align}
We can extend the functional $E_\lambda$ defined in
\eqref{problemintro} over $\mathcal S_m$ to the whole of $\K_m$ by
setting
\begin{equation}\label{problem}
  E_\lambda(\Om) := P(\Omega) + \lambda\,
  \I_1(\Om) + \int_{\Om} g(x)\,dx,
 \qquad \Omega\in\K_m\,, 
\end{equation}
where $P(\Omega)$ denotes the De Giorgi perimeter of $\Omega$, defined as
\begin{equation}\label{defper}
  P(\Omega):=\sup\left\{\int_{\Omega}{\rm div} \, \phi \, \, dx\,:\,\phi\in
    C^{1}_c(\mathbb R^2; \mathbb R^2),\, \|\phi\|_{L^\infty(\mathbb
        R^2)}\le 1 \right\} ,
\end{equation}
which coincides with $\mathcal H^1(\partial \Omega)$ if $\Omega\in\S_m$.
Given an open set $A\subset \R^2$, we also define the perimeter of
$\Omega$ in $A$ as
\begin{align}
  P(\Omega;A):=\sup\left\{\int_{{\Omega \cap A}}{\rm div} \, \phi
    \, \, dx\,:\,\phi\in 
    C^{1}_c(A; \mathbb R^2),\, \|\phi\|_{L^\infty(A)}\le 1 \right\} ,
\end{align}
so that, in particular, $P(\Om)=P(\Om;\R^2)$.

For $\Omega\in \A_m$ we introduce the $L^1$-relaxed energy for
$E_\lambda$ restricted to $\S_m$:
\begin{equation}
  \label{oE}
  \oE(\Om) := \inf_{\Omega_n\in\S_m,\,|\Om_n\Delta \Om|\to 0} \
  \liminf_{n\to \infty} E_\lambda(\Omega_n). 
\end{equation}
We observe that, as a consequence of Proposition \ref{p:schmidt} and
Corollary \ref{cor:schmidt} in the following section, we can
equivalently define $\oE$ starting from sets $\Omega_n\in\K_m$ in
\eqref{oE}, that is, there holds
\begin{align}
\oE(\Om) = \inf_{\Omega_n\in\K_m,\,|\Om_n\Delta \Om|\to 0} \
\liminf_{n\to \infty} E_\lambda(\Omega_n).
\end{align}

Our first result below provides an explicit characterization of the
relaxed energy for sets in $\K_m$.

\begin{theorem}\label{teorel}
  Let $g$ be a continuous function bounded from below and let $m >
  0$. Then for any $\Om\in \K_m$ {we have
    $\oE(\Om) = \E_\lambda(\Om)$, where}
  \begin{equation} \label{Erelaxeddef} {\E_\lambda}(\Om)
    := \begin{cases}
      E_\lambda(\Om) &  \text{if }\lambda\le\lambda_\Om,
      \\
      E_{\lambda_\Om}(\Om) + 2\pi \big(\sqrt\lambda - \sqrt\lambda_\Om\big) 
      & \text{if }\lambda>\lambda_\Om,
    \end{cases}
  \end{equation}
  {and}
  \begin{equation}\label{lambda}
    \lambda_\Om := \left(\frac{\pi}{\I_1(\Om)}\right)^2.
  \end{equation}
  In particular, $E_\lambda$ is lower semicontinuous on $\K_m$ with
  respect to the $L^1-$convergence if and only if
  $\lambda\le \lambda_c(m)$.
\end{theorem}

\noindent We recall that the quantity $\lambda_c(m)$ is defined in
  \eqref{eq:lamc}.

Note that as can be easily seen from the definition of
$\lambda_\Omega$, we always have
$E_\lambda(\Om) \geq E_{\lambda_\Om}(\Om) + 2 \pi (\sqrt{\lambda} -
\sqrt{\lambda_\Om})$ for all $\lambda > 0$. Therefore, the result of
Theorem \ref{teorel} may be interpreted as follows: either it is
energetically convenient to distribute all the charges over the set
$\Omega$ or it is favorable to send some excess charge off to
infinity. More precisely, for a given set $\Omega$ such that
$\lambda > \lambda_\Omega$ it is possible to find a sequence of sets
converging to $\Omega$ in the $L^1$ sense that contain vanishing parts
with positive capacitary energy. In particular, the vanishing parts
contribute a finite amount of energy to the limit, which is a
non-trivial property of the considered problem.
      
The above result implies existence of minimizers for
$\overline E_\lambda$ in $\mathcal A_m$, as long as we require 
the coercivity and the local Lipschitz continuity of the function $g$.
\begin{definition}\label{defGc}
  We say that a function $g:\R^2\to \R$ is coercive if
\begin{align}
  \label{coerc}
  \lim_{|x| \to +\infty} g(x) = +\infty,
\end{align}
Furthermore, we define the class of functions $\mathcal G$ as
follows:
\begin{align}
  \label{eq:G}
  \mathcal G := \{g:\R^2\to [0,+\infty)\,:\;
    \text{$g$ is locally Lipschitz continuous and coercive}\}.
\end{align}
\end{definition}

\noindent Note that the assumption of positivity of $g$ in
\eqref{eq:G} is not essential and may be replaced by boundedness of
$g$ from below. For this class of functions, which represent the
effect of confinement by an external potential $g$, we have the
following existence result.

\begin{theorem}\label{teoexistence}
  Let $m > 0$, let $\lambda<\lambda_c(m)$ and let
  $g\in \mathcal G$.  Then there exists a minimizer $\Omega_\lambda$
  for $E_\lambda$ over all sets in $\mathcal K_m$.
\end{theorem}

We stress that the existence result stated in Theorem
\ref{teoexistence} is not a direct consequence of Theorem
\ref{teorel}, the reason being that the class $\mathcal K_m$ is not
closed under $L^1-$convergence, and is in fact one of the main results
of this paper.

Given $\Omega\in\K_m$, we let $\Omega^+$ be defined as
\begin{align}
\Omega^+ := \left\{ x\in\Omega:\ |\Omega\cap B_r(x) |>0\text{ for all }r>0\right\} .
\end{align}
Notice that $\Omega^+$ is a closed set.  Indeed, recalling that $\Om$
is closed we have that $x\in (\Om^+)^c$ if and only if there exists
$r>0$ such that $ |\Omega\cap B_r(x)|=0$. Then, for every
$y\in B_r(x)$ there holds $|\Om\cap B_{r-|y-x|}(y)|=0$, hence
$y\in (\Om^+)^c$, that is, $(\Om^+)^c$ is open and $\Om^+$ is closed.
Furthermore, if $\Om\in\K_m$, we have $\Omega^+\subset\Omega$ and
$\Omega\setminus\Omega^+= \left\{ x\in\Omega:\ |\Omega\cap B_r(x)
  |=0\text{ for some }r>0\right\} \subset\partial \Omega$, so that
\begin{align}
\mathcal H^1(\Omega\setminus\Omega^+)\le \mathcal H^1(\partial
\Omega)<+\infty.
\end{align}
As a consequence we get $|\Om|=|\Om^+|$, $P(\Om)=P(\Om^+)$. Moreover,
since the Hausdorff dimension of $\Omega\setminus\Omega^+$ is at most
$1$, then $\I_1(\Om)=\I_1(\Om^+)$ (see Lemma \ref{l:optm} below).
Therefore $E_\lambda(\Om)=E_\lambda(\Om^+)$, and $\Om$ is a minimizer
of $E_\lambda$ if and only if $\Om^+$ is a minimizer.  We observe that
$\Omega^+$ is a representative of $\Omega$ which is in general more
regular, and for which we can show density estimates which do not
necessarily hold for $\Omega$ itself.

We now state a partial regularity result for the minimizers given in
Theorem \ref{teoexistence}.

\begin{theorem}\label{teodensity}
  Let $m>0$, $\lambda<\lambda_c(m)$ and $g\in\mathcal G$.  Let also
  {$\Omega_\lambda$} be a minimizer of $E_\lambda$ {over
    $\mathcal K_m$}.  Then there exist $c > 0$ universal and $r_0>0$
  depending only on $m$, $\lambda$ and $g$ such that for every
  $0<r\le r_0$ and every $x\in\partial\Omega_\lambda^+$ there holds
  \begin{equation}\label{eqdens}
    |\Om_\lambda\cap B_r(x) |\ge c {\left( 1 - {\lambda \over
          \lambda_c(m)} \right)^2} r^2 \qquad {\rm and}\qquad
    |\Om_\lambda^c\cap B_r(x) |\ge c {\left( 1 - {\lambda \over
          \lambda_c(m)} \right)^2} r^2. 
  \end{equation}
  Furthermore, both $\Omega_\lambda$ and $\Omega_\lambda^c$ have a
  finite number of indecomposable components in the sense of
  \cite[Section 4]{ACMM}.
\end{theorem}

\begin{remark}\label{remamb}\rm
From Theorem \ref{teodensity} and \cite[Theorem II.5.14]{maggi} it follows that
\begin{align}
\mathcal H^1(\partial\Om_\lambda^+)=P(\Om_\lambda^+)=P(\Om_\lambda)\,.
\end{align}
Therefore, the set $\Omega_\lambda^+$ also minimizes the energy $E_\lambda$
  as defined in \eqref{problemintro}, among all sets in
  $\mathcal K_m$.
\end{remark}

The semicontinuity of $E_\lambda$ allows us to get existence of
minimizers for $\lambda< \lambda_c(m)$, but we cannot say much about
their qualitative shape, besides the partial regularity result given
in Theorem \ref{teodensity}. On the other hand, for $m$ sufficiently
small and $\lambda$ small relative to $\lambda_c(m)$ we can show
that the minimizers become close to a single ball of mass $m$ located
at a minimum of $g$.

\begin{theorem}\label{teoasy}
  Let $g\in\mathcal G$ and let $m_k,\,\lambda_k>0$, $k\in\mathbb N$,
  be two sequences such that
	\begin{align}\lim_{k\to +\infty}m_k=0\quad \text{ and }\quad
          \limsup_{k\to +\infty}\frac{\lambda_k}{\lambda_c(m_k)} < 1
          \,.
        \end{align}  
	Then the following assertions are true:
	\begin{enumerate}
        \item For every $k$ large enough there exists a minimizer
            $\Omega_{k}$ of $E_{\lambda_k}$ over $\mathcal K_{m_k}$.
          \item As $k\to\infty$, there exists a bounded sequence
            $(x_k) \in \R^2$ such that the translated and rescaled
            minimizers
            $\left( \frac{\pi}{m_{k}} \right)^{\frac12} (\Omega_{k} -
            x_k)$ converge to $B_1(0)$ in the Hausdorff distance.
          \item If $x_0$ is a cluster point of $(x_k)$, then $x_0 \in \mathrm{argmin} \, g$.
       \end{enumerate}
     \end{theorem}

     We note that in the local setting, i.e., when $\lambda = 0$, the
     result in Theorem \ref{teoasy} was obtained by Figalli and Maggi
     in \cite{FigMag}, who in fact also obtained strong quantitiative
     estimates of the rate of convergence of these minimizers to balls
     in this perimeter-dominated regime. This is made possible in the
     context of local isoperimetric problems with confining potentials
     by an extensive use of the regularity theory available for such
     problems \cite{maggi}. In contrast, minimizers of our problem
     fail to be quasi-minimizers of the perimeter and, therefore,
     their $C^{1,\alpha}$-regularity is a difficult open question. The
     proof of Theorem \ref{teoasy}, which extends some results of
     \cite{FigMag} to the nonlocal setting involving capacitary
     energies relies on the arguments used to obtain partial
     regularity of the minimizers in the subcritical regime in Theorem
     \ref{teodensity}. These estimates are also the first step towards
     the full regularity theory of the minimizers of $E_\lambda$.

Finally,  we derive the Euler--Lagrange equation for the energy
$E_\lambda$ under some smoothness assumptions on the shape of the
minimizer. The main issue here is to compute the first variation of
the functional $\I_1(\Omega)$ with respect to the deformations of the
set $\Omega$. To that end, given a compact set $\Omega$ with a
sufficiently smooth boundary, we introduce the potential function
\begin{equation}
  \label{vOmdef}
  v_\Om(x) := \int_\Omega {d \mu_\Om(y) \over |x - y|}, 
\end{equation}
where $\mu_\Om$ is the equilibrium measure of $\Om$ minimizing
$\mathcal I_1$.  The normal $\frac12$-derivative of the potential of
$v_\Omega$ at the boundary of $\Omega$ is then defined as
\begin{equation}\label{derivatafrazionaria}
  \partial^{1/2}_{\nu}v_\Om(x):=\lim_{s\to0^+}\frac{v_\Om(x+s\nu(x)) -
    v_\Omega(x)}{s^{1/2}},
\end{equation}
where $x \in \partial \Omega$ and $\nu(x)$ is the outward normal
vector to $\partial\Om$ at $x$.  

\begin{theorem}\label{teocinque}
  Let $\Omega$ be a compact set with boundary of class $C^2$, let
  $\zeta \in C^\infty(\R^2, \R^2)$, and let $(\Phi_t)_{t \in \R}$ be a
  smooth family of diffeomorphisms of the plane satisfying
  $\Phi_0 = \mathrm{Id}$ and
  $\left. {d \over dt} \Phi_t \right|_{t=0} = \zeta$. Then the normal
  $\frac 12$-derivative $\partial^{1/2}_\nu v_\Om$ of the potential
  $v_\Omega$ from \eqref{vOmdef} is well-defined and belongs to
  $C^\beta(\partial\Omega)$ for any $\beta\in (0,1/2)$.  Moreover, we
  have
	\begin{equation}
          \left. \frac{d}{dt}{\mathcal
              I}_1(\Phi_t(\Omega)) \right|_{t=0} = 
          -\frac{1}{8}\,\int_{\partial\Omega}(\partial^{1/2}_\nu
          v_\Omega(x))^2\zeta(x)\cdot \nu(x)\,d\mathcal H^{1}(x). 
	\end{equation}
	As a consequence, the Euler--Lagrange equation for a
          critical point $\Omega \in \mathcal A_m$ of $E_\lambda$
          satisfying the above smoothness conditions is
	\begin{equation}\label{eqEL}
\kappa - \frac{\lambda}{8} \, (\partial^{1/2}_\nu v_\Omega)^2 + g= p
\qquad {\rm on}\ \partial\Om,
\end{equation}
where $\kappa$ is the curvature of $\partial \Om$ (positive if
  $\Omega$ is convex) and $p\in\R$ is a Lagrange multiplier due to
the mass constraint.
\end{theorem}

We note that the result in Theorem \ref{teocinque} relies on recent
regularity estimates for fractional elliptic PDEs obtained in
\cite{ro,ro-s,DS}. It is also closely related to the result of
Dalibard and G\'erard-Varet \cite{DGV13} on the shape derivative of a
fractional shape optimization problem.

\section{Preliminaries: capacitary estimates, perimeters and connected components}\label{preliminari}

In this section we give some preliminary definitions and results
about the functionals $\I_1$ and $P$ that define $\E_\lambda$. We
begin with an important remark about the necessity of introducing the
classes $\K_m$ and $\S_m$.

\begin{remark}{\rm As mentioned in the Introduction, we have to choose
    carefully the admissible class for the minimization of
    $E_\lambda$. A natural choice would be minimizing $E_\lambda$ in
    the class of finite perimeter sets.  However, in this class the
    functional $E_\lambda$ is never lower semicontinuous.  Indeed,
    given a set $\Omega\subset\R^2$ and $\eps>0$ it is possible to
    find another set $\Om_\eps$, with $|\Om\Delta\Om_\eps|=0$ and
    $P(\Om)=P(\Om_\eps)$, but with $\I_1(\Om_\eps)<\eps$ (see the
    Introduction in \cite{MurNovRuf}).  Such a construction cannot be
    accomplished in $\K_m$. In this sense, $\K_m$ is the largest class
    in which it is meaningful to consider the minimization of
    $E_\lambda$.  }
\end{remark}

In \cite[Theorems 1 and 2]{MurNovRuf} uniform bounds on
$E_\lambda(\Om)$ were proved for $g=0$, which are attained on balls.
These estimates will play a crucial role in the proof of Theorem
$\ref{teorel}$, and we recall them in the following lemma.

\begin{lemma}\label{lemtom}
	For any $\Om\in\K_m$ there holds
	\begin{equation}\label{universal}
	\mathcal H^1(\partial\Om)+\lambda\I_1(\Om)\ge 2\pi\sqrt\lambda.
	\end{equation}
	Moreover, if $\lambda\le\lambda_c(m)$ there also holds
        \begin{align}
          \label{eq:HIHIB}
          \mathcal H^1(\partial\Om)+\lambda\I_1(\Om)\ge \mathcal
          H^1(\partial
          B_r(x_0))+\lambda\I_1(B_r(x_0)),          
        \end{align}
	where $r = \sqrt{m/\pi}$ and $x_0 \in \R^2$, i.e.,
          $B_r(x_0)$ is a ball of measure $m$, and the equality holds
        if and only of $\Om=B_r(x_0)$ for some $x_0 \in \R^2$.
\end{lemma}
We now recall some basic facts about the functional $\I_1$.

\begin{lemma}{\cite[Lemma 1]{MurNovRuf}}
	\label{l:optm}
	Let $\Omega\subset\R^2$ be a compact set such that
          $|\Omega| > 0$ and $\mathcal H^1(\partial \Omega)<+\infty$.
        Then there exists a unique probability measure $\mu$ {over
          $\R^2$} supported on $\Omega$ such that
	\begin{align}
	\I_1(\Omega) = \int_\Omega \int_\Omega {d \mu(x)  d \mu(y) \over
		|x - y|}.
	\end{align}
	Furthermore, $\mu(\partial \Omega) = 0$, and we have
	$d \mu(x) = \rho(x) dx$ for some
	$\rho \in L^1(\Omega)$ satisfying
	$0 < \rho(x) \leq C / \mathrm{dist} \, (x, \partial \Omega)$ for
	some constant $C > 0$ and all $x \in \mathrm{int}(\Omega)$.
\end{lemma}

Another useful estimate is the following:

\begin{lemma}{\cite[Lemma 2]{MurNovRuf}}\label{lemdue}
  Let $\Omega_1,\Omega_2\subset\R^2$ be compact sets with positive
  measure such that $\mathcal H^1(\partial \Omega_i)<+\infty$ for
  $i\in \{1,2\}$, and $|\Omega_1\cap \Omega_2|=0$.  Then, for all
  $t\in [0,1]$ there holds
	\begin{align}
	\I_1(\Omega_1\cup\Omega_2) \le t^2 \I_1
	(\Omega_1) + (1-t)^2 \I_1(\Omega_2)
	+ \frac{2\,t\,(1-t)}{{\rm dist}(\Om_1,\Om_2)}\,,
	\end{align}
	and there exists $\bar t\in (0,1)$ such that 
	\begin{align}
	\I_1(\Omega_1\cup\Omega_2) > \bar t^2 \I_1(\Omega_1) + (1-\bar t)^2
	\I_1(\Omega_2).
	\end{align}
\end{lemma}

	From \cite[Section 2]{gnrI} (see also \cite{landkof}) we have that
	\begin{equation}\label{capuno}
	\I_1(\Om)=\frac{2\pi}{{\rm cap}_1(\Om)}\,,
	\end{equation}
	whenever $\Om$ is a compact set, where 
	${\rm cap}_1(\Om)$ is the $\frac 12$-capacity of $\Om$ defined as
	\begin{equation}\label{capcap}
          {\rm
            cap}_1(\Omega):=\inf\left\{{\|u\|^2_{\mathring{H}^{\frac12}(\R^2)}}
            :u\in  C^1_c(\R^2),\quad u\ge\chi_\Omega\right\}, 
	\end{equation}
	and
	\begin{align}
          {\|u\|^2_{\mathring{H}^{\frac12}(\R^2)} } :=\frac{1}{4\pi}
          \int_{\R^2}\int_{\R^2} \frac{|u(x)-u(y)|^2}{|x-y|^3}\,dxdy
	\end{align}
        {is the Gagliardo norm of the homogeneous fractional Sobolev
          space obtained via completion of $C^\infty_c(\R^2)$ with
          respect to that norm \cite{DPV,LMM}}.  For the sake of
        completeness, we provide a short justification of this fact:
        Let $v_\Om:=\mu_\Omega*|\cdot|^{-1}$ be the potential of
        $\Omega$, where $\mu_\Om$ is the equilibrium measure for
        $\Omega$. Then $v_\Om$ satisfies (see, for instance,
        \cite[Lemma 2.11]{gnrI})
	\begin{align}
	\begin{cases}
	(-\Delta)^\frac12 v_\Om=0\qquad&\text{on }\Omega^c\\
	v_\Om=\I_1(\Omega)&\text{a.e. on }\Omega\\
	\lim_{|x|\to+\infty}v_\Om(x)=0\,.
	\end{cases}
	\end{align}
	Furthermore,
        \begin{align}
          \label{uOm}
          u_\Om:=\I_1^{-1}(\Omega) \, v_\Om
        \end{align}
        is the $\frac12$-capacitary potential of $\Omega$ attaining
        the infimum in \eqref{capcap}.  In particular, by
        \cite[Proposition 3.4]{DPV} we have
	\begin{align}
	\I_1(\Omega)=\int_{\R^2} v_\Om\,d\mu_\Om =
        \frac{1}{2\pi}\,{\|
          v_\Om\|^2_{{\mathring{H}}^\frac12(\R^2)}} \,,
	\end{align} 
	so that
	\begin{equation}\label{1su1}
          {\rm
            cap}_1(\Omega)={
            \|u_\Om\|^2_{{\mathring{H}}^\frac12(\R^2)}} = 
          \frac{{\|v_\Om\|^2_{{\mathring{H}}^\frac12(\R^2)}}
          }{\I_1(\Omega)^2}=   
          \frac{2\pi}{\I_1(\Omega)}. 
	\end{equation}  
	The link with the classical Newtonian capacity ${\rm
          cap}(\Om)$, defined for $\Om\subset\R^3$ as
	\begin{align}
          {\rm cap}(\Om) := \inf\left\{\|{\nabla u}\|^2_{L^2(\R^3)}
          \,:\,u\in C^1_c(\R^3),\quad u\ge\chi_{\Om} \right\},
	\end{align}  
	is given by the equality
        \begin{equation}\label{capcapo}
        {\rm cap}_1(\Omega)= 2\,{\rm cap}(\Omega\times\{0\}),
        \end{equation}
        for any compact set $\Omega\subset\R^2$ \cite[Theorem
        11.16]{liebloss}.  Finally, we recall that 
          ${\rm cap}_1(\Omega) = 0$ if $\mathcal H^1(\Omega) < \infty$
          (see \cite[Theorem 3.14]{landkof}).

        We note that a priori the functional $\I_1$ is not lower
        semicontinuous with respect to $L^1$ convergence.  However,
        given a compact set $\Omega$, $\I_1$ is semicontinuous along a
        specific family of sets, namely sets of the form
        \begin{align}
          \label{Omdelta}
          \Omega^\delta := \{x\in\R^2:\ {\rm dist}(x,\Om)\le\delta\},          
        \end{align}        
        for $\delta\to0$.  This is formalized in the next lemma, and
        then exploited in Proposition \ref{teorelpro}.

\begin{lemma}\label{1maggio}
  Let $\Omega$ be a compact subset of $\R^2$ and let
  $(\delta_n)_{n\in\mathbb N}\subset[0,+\infty)$ and
  $\bar \delta\in[0,+\infty)$ be such that
  $\delta_n\to\overline\delta$ {as $n \to \infty$}.  
   Then
	\begin{align}
	\I_1(\Omega^{\overline{\delta}})\le \liminf_{n\to+\infty}\I_1(\Omega^{\delta_n}).
	\end{align}
	Moreover, if $\delta_n\searrow\overline{\delta}$ there holds
	\begin{align}
	\I_1(\Omega^{\overline{\delta}})= \lim_{n\to+\infty}\I_1(\Omega^{\delta_n}).
	\end{align}
\end{lemma}

\begin{proof}
  We can suppose that $(\delta_n)_{n\in\mathbb N}$ is a monotone
  sequence. We have two cases: If $\delta_n\nearrow\overline{\delta}$,
  then $\delta_n\le\overline{\delta}$ for any $n$ and thus by the
  monotonicity of $\I_1$ with respect to {set inclusions}, we have
  that $\I_1(\Omega^{\delta_n})\ge\I_1(\Omega^{\overline{\delta}})$
  and the lower semicontinuity is proven.
	
  We deal now with the case $\delta_n\searrow\overline{\delta}$.  Let
  us fix $\varepsilon>0$ and let $\varphi\in C^1_c(\R^2)$ be such that
  $\varphi>\chi_{\Omega^{\overline{\delta}}}$, and
  $ {\|\varphi\|^2_{{\mathring{H}}^{\frac12}(\R^2)}}\le {\rm
    cap}_1(\Omega^{\overline{\delta}})+\varepsilon$.  Then, since
  $\{ \varphi>1\}$ is an open set which contains
  $\Omega^{\overline{\delta}}$, for $n$ big enough (depending on
  $\varepsilon$) $ \varphi$ is also a test function for
  ${\rm cap}_1(\Omega^{\delta_n})$, and we get
	\begin{align}
	{\rm cap}_1(\Omega^{\overline\delta})\le
	{\rm cap}_1(\Omega^{\delta_n})
	\le 
	{\|\varphi\|^2_{{\mathring{H}}^{\frac12}(\R^2)}}
	\le
	{\rm cap}_1(\Omega^{\overline{\delta}})+\varepsilon.
	\end{align}
	Letting $\varepsilon\searrow 0$, we get the continuity of
        ${\rm cap}_1$ and hence of $\I_1$ by \eqref{1su1}.
\end{proof}

We prove now a result which turns out to be very useful in the proof
of the semicontinuity result in Theorem \ref{teorel}, as well as of
the existence and regularity results in Theorems \ref{teoexistence}
and \ref{teodensity}.

\begin{lemma}\label{distributionofcharges}
    Let $\Omega=U\cup V$ with $U$ and $V$
  {compact} sets of finite positive measure and such that
  $|U\cap V|=0$. Then we have
	\begin{align}
	\I_1(\Omega)\,\ge\, \I_1(U) -\frac{\pi}{4{|U|}{}}\,P(V).
	\end{align}
\end{lemma}

\begin{proof}
	By Lemma \ref{lemdue} we have
	\begin{align}
	\I_1(\Omega)\ge \min_{t\in[0,1]}\left\{t^2\I_1(U)+(1-t)^2\I_1(V)\right\}.
	\end{align}
	By computing 
	the minimum on the right-hand side, we get
	\begin{equation}\label{fe}
	\I_1(\Omega)\ge \I_1(U)-\frac{\I_1^2(U)}{\I_1(U)+\I_1(V)}
	\ge \I_1(U)-\frac{\I_1^2(U)}{\I_1(V)}.
	\end{equation}
	We recall that $\I_1$ is maximized by the ball among sets of
        fixed volume.  Letting $B:={B_{\sqrt{|U|/\pi}}(0)}$, we then
        get that $|B|=|U|$ and {\cite[Lemma 3.3]{MurNovRuf}}
	\begin{equation}\label{bm}
	\I_1(U)\le\I_1(B)=\frac{\pi^{\frac32}}{2\sqrt{|U|}}.
	\end{equation}
	Moreover, by \cite[Corollary 3.2]{NovRuf} the dilation
        invariant functional ${\mathcal F := \I_1(\cdot) P(\cdot)}$ is
        minimized by balls,
	and on a ball $B_r$ it takes the value
        $\I_1(B_r)P(B_r)=\pi^2$, so that 
	\begin{align}
	\I_1(V)\ge \frac{\pi^2}{P(V)}.
	\end{align}
	We plug these two estimates into  \eqref{fe} to get
	\begin{align}
          \I_1(\Omega)\ge
          \I_1(U)-\left(\frac{\pi^\frac32}{2{\sqrt{|U|}}{}}\right)^2\frac{P(V)}{\pi^2}
          = \I_1(U) -\frac{\pi}{4{|U|}{}}\,P(V),
	\end{align}
        which is the desired estimate.
\end{proof}

In the proof of Theorem \ref{teoexistence} we shall use some
topological features of sets of finite perimeter in dimension
two. Since these sets are defined in the $L^1-$sense ({as equivalence
  classes}), it is not a priori immediate how to define what a
connected component for a set of finite perimeter is. A suitable
notion of connected components for sets of finite perimeter was
introduced in \cite{ACMM}. Below we recall some of their main features
that we shall use in the sequel.

Given $\Omega \in \mathcal K_m$, let $\mathring \Omega^M$ be its
measure theoretic interior, namely:
\begin{align}
\label{eq:OmMbar}
  \mathring \Omega^M := \left\{ x \in \R^2 \ : \ \lim_{r \to 0}
  {|\Omega \cap B_r(x)| \over \pi r^2} =1 \right\}.
\end{align}
Since
$P(\mathring{\Omega}^M)= P(\Omega) = \mathcal H^1(\partial^M \Omega)
\leq \mathcal H^1(\partial \Omega) < +\infty$, where
$\partial^M \Omega$ is the essential boundary of $\Omega$
\cite{maggi}, the set $\mathring \Omega^M$ is a set of finite
perimeter.  Therefore, following \cite{ACMM}, there exists an at most
countable family of sets of finite perimeter $\Omega_i$ such that
$\mathring \Omega^M = \left( \bigcup_{{i}}\mathring\Omega_i^M\right)
\cup \Sigma$, with $\mathcal H^1(\Sigma) = 0$, where the sets
$\mathring\Omega_i^M$ are the so-called {\it indecomposable
  components} of $\mathring\Omega^M$. In particular, the sets
$\Omega_i$ admit unique representatives that are connected and satisfy
the following properties:
\begin{enumerate}[ \ (i)]
	\item $\mathcal H^1(\Omega_i\cap \Omega_j)=0$ for $i\ne j$, 
	\item $|\Omega|=\sum_{{i}} |\Omega_i|$,
	\item $P(\Omega) =\sum_{{i}}P(\Omega_i)$,
        \item $\Omega_i=\overline{\mathring\Omega_i^M}$.
\end{enumerate} 
Moreover, each set $\Omega_i$ is indecomposable in the sense that it
cannot be further decomposed as above. We refer to these
representatives of $\Omega_i$ as the {\em connected components} of
$\Omega$.  We point out that this notion coincides with the standard
notion of connected components in the following sense: if $\Omega$ has
a regular boundary (Lipschitz continuous being enough) then the
components $\Omega_i$ are the closures of the usual connected
components of the interior of $\Omega$.

Such a representation of $\Omega$ as a union of connected components
allows us to convexify the components in order to decrease the
energy. Indeed, for every $i \in \mathbb N$ there holds
$\mathcal I_1(\Omega_i) \leq \mathcal I_1(co(\Omega_i))$ and
$\mathcal H^1(\partial \, co(\Omega_i))\le \mathcal H^1(\partial
\Omega_i)$, where $co(\Omega_i)$ denotes the convex envelope of the
component $\Omega_i$. This follows from the fact that
$\Omega_i \subseteq co(\Omega_i)$, and that the outer boundary of a
connected component can be parametrized by a Jordan curve of finite
length (see \cite[Section 8]{ACMM}). In addition, since
$\partial \Omega$ is negligible with respect to the equilibrium
measure for $\I_1(\Omega)$ by Lemma \ref{l:optm}, we have
$\I_1(\Omega) = \I_1(\mathring \Omega^M)$.

The next result shows that the relaxations of $E_\lambda$ in  $\S_m$ and in $\K_m$ coincide.

\begin{proposition}\label{p:schmidt}
	Given $\Omega \in \K_m$, there exists a sequence of sets
	$\Omega_n \in \S_m$ such that
	\begin{align}
	\label{limsupschmidt}
	\lim_{n \to \infty}|\Omega_n \Delta \Omega|=0\qquad \text{ and }\qquad
	\limsup_{n \to \infty} E_\lambda(\Omega_n) \leq
	E_\lambda(\Omega). 
	\end{align}
\end{proposition}

\begin{proof}
  Assume first that $P(\Omega)=\mathcal H^1(\partial\Omega)$. Then by
  \cite[Theorem 1.1]{schmidt} applied to $B_R(0) \setminus \Omega$,
  for $R > 0$ big enough there exists a sequence of compact sets
  $\widetilde \Omega_n$ with smooth boundaries such that
  $\widetilde \Omega_n \supset \Omega$,
  $| \widetilde \Omega_n \Delta \Omega| \to 0$ and
  $P(\widetilde \Omega_n) \to P(\Omega)$ as $n \to
  \infty$. Furthermore, by monotonicity of $\mathcal I_1$ with respect
  to set inclusions we have
  $\mathcal I_1(\widetilde \Omega_n) \leq \mathcal I_1(\Omega)$. Now,
  we define
  $\Omega_n := (m / |\widetilde \Omega_n|)^{1/2} \widetilde \Omega_n
  \in \S_m$, and in view of the fact that
  $ |\widetilde \Omega_n| \to m$ as $n \to \infty$ we obtain the
  result.
	
  Let us now consider {the} general case.  By \cite[Corollary
  1]{ACMM}, there exists a sequence of sets $\Omega_n\in \K_m$ such
  that { $\partial \Omega_n$ is a finite union of Jordan curves, and
    as $n\to \infty$ we have}
	\begin{align}
	| \Omega_n \Delta \Omega| \to 0, \qquad P(\Omega_n)\to
        P(\Omega),\qquad P(\Omega\setminus\Omega_n)\to 0.
	\end{align}
	{In particular,} $P(\Omega_n)=\mathcal H^1(\partial
        \Omega_n)$ for every $n\in\mathbb N$.  {Then b}y Lemma
        \ref{distributionofcharges} it follows that
	\begin{align}
	\I_1(\Omega_n)\le \I_1(\Omega_n\cap\Omega)\le \I_1(\Omega)
        + \omega_n, 
	\end{align} 
	with 
	\begin{align}
	\omega_n:= \frac{\pi}{4\,|\Omega_n\cap\Omega|}\,P(\Omega\setminus\Omega_n)
	\to0\qquad \text{ as $n\to \infty$,}
	\end{align}
	so that 
	\begin{align}
	\limsup_{n\to+\infty} E_\lambda(\Omega_n)\le E_\lambda(\Om).
	\end{align} 
	Applying now the approximation with regular sets  to each set $\Omega_n$,  we conclude by a diagonal argument.
\end{proof}

Proposition \ref{p:schmidt} yields the following characterization of the relaxed energy $\oE$.

\begin{corollary}\label{cor:schmidt}
  For every $\Omega\in \A_m$ there holds
	\begin{equation}
	\oE(\Om) = \inf_{\Omega_n\in\K_m,\,|\Om_n\Delta \Om|\to 0} \
	\liminf_{n\to \infty} E_\lambda(\Omega_n). 
	\end{equation}
\end{corollary}
 
\section{The relaxed energy: Proof of Theorem \ref{teorel}}\label{teo1}

In this section we prove Theorem \ref{teorel}. We divide the
proof into first characterizing the relaxation of $E_\lambda$ in
Proposition \ref{teorelpro} and then showing the semicontinuity of
$E_\lambda$ for $\lambda \leq \lambda_c(m)$ in Proposition
\ref{correl}.

\begin{proposition}\label{teorelpro}
  For any $\Om\in \K_m$, there holds
\begin{equation}
\oE(\Om) =\begin{cases}
E_\lambda(\Om) &  \text{if }\lambda\le\lambda_\Om,
\\
E_{\lambda_\Om}(\Om) + 2\pi \big(\sqrt\lambda - \sqrt\lambda_\Om\big) 
& \text{if }\lambda>\lambda_\Om,
\end{cases}
\end{equation}
where $\lambda_\Om$ is defined in \eqref{lambda}.
\end{proposition}

\begin{proof}
  Let $\Om_n$ be a sequence of sets in $\S_m$ such that
  $|\Om_n\Delta\Om|\to 0$ as $n\to \infty$.  For any $\delta>0$ we let
  $\Om^\delta$ as in \eqref{Omdelta}.  Notice that there exists
  $\delta_0>0$ such that $\Omega^\delta\in\K_{m+\omega(\delta)}$, for
  any $\delta\le \delta_0$, where $\omega(\delta)\to0$ as $\delta\to0$
  by the monotone convergence theorem.

  For any $n\in\mathbb N$ we let
  $\Om_n(\delta) {:=} \Om_n\cap \Om^\delta$ and
  $\widetilde\Om_n(\delta) {:=}
  \overline{\Omega_n\setminus\Om^\delta}$.  By \cite[Section II.7.1]{maggi}, we have
  \begin{equation}\label{pom}
    P(\Omega_n) \ge P(\Omega_n; {\rm int}( \Om^\delta)) +  P(\Omega_n; \R^2\setminus\Om^\delta)
    = P(\Om_n(\delta))+P(\widetilde\Om_n(\delta))
    -2\mathcal H^1(\Om_n \cap \partial \Om^\delta).    
\end{equation}
Notice that for any fixed $\delta\in (0,\delta_0)$, by Coarea Formula
\cite[Theorem 18.1]{maggi} we also have
\begin{align}
\begin{aligned}
  \int_0^{\delta}\mathcal H^1(\Om_n \cap \partial \Om^t)\,dt =
  |\Omega_n\cap (\Om^{\delta}\setminus \Omega)|\le |\Om_n\Delta\Om|,
\end{aligned}
\end{align}
Therefore we can choose $\delta_n\in (\delta/2,\delta)$ such that
\begin{align}
\mathcal H^1(\Om_n \cap \partial \Om^{\delta_n})\le \frac{2\,|\Om_n\Delta \Om|}{\delta}.
\end{align}  
Recalling \eqref{pom} this gives
\begin{equation}\label{tom}
  P(\Omega_n) \ge
  P(\Om_n({\delta_n}))+P(\widetilde\Om_n({\delta_n}))-\omega^\delta_n, 
\end{equation}
where $\omega^\delta_n \le \frac{4}{\delta}|\Om_n\Delta \Om|$.  Up to a
subsequence, we can assume that $\delta_n\to \bar \delta$ as
$n\to\infty$ for some $\bar\delta \in [\delta/2,\delta]$. Moreover, we
can choose $\delta_n$ such that
$P(\widetilde\Om_n({\delta_n}))= \mathcal H^1(\partial
\widetilde\Om_n({\delta_n}))$ (see \cite[Equation (68)]{mennucci}).

We now estimate the nonlocal term. Since
$\Omega_n\subset \Om_n(\delta_n) \cup \widetilde\Omega_n({\delta_n})$,
we have
\begin{equation}\label{tim}
\begin{aligned}
\I_1(\Omega_n) &\ge \I_1(\Om_n(\delta_n) \cup \widetilde\Omega_n({\delta_n}))\\
&\ge \min_{t\in [0,1]}\left(t^2\I_1(\Om_n(\delta_n)) + (1-t)^2\I_1(\widetilde\Omega_n(\delta_n))\right)\\
&\ge \min_{t\in [0,1]}\left(t^2\I_1(\Om^{\delta_n}) + (1-t)^2\I_1(\widetilde\Omega_n(\delta_n))\right),
\end{aligned}
\end{equation}
where the second inequality follows from Lemma \ref{lemdue}, while the
third is due to the fact that $\Omega_n(\delta_n)$ is contained in
$\Omega^{\delta_n}$ and that $\I_1$ is decreasing with respect to
{set inclusions}.

By Lemma \ref{lemtom} we have that
\begin{align}
  \mathcal H^1(\partial \widetilde\Om_n({\delta_n})) +\lambda
  (1-t)^2\I_1(\widetilde\Omega_n({\delta_n}))\ge
  2\pi(1-t)\sqrt{\lambda}\,.
\end{align}
Thus, by combining \eqref{tim} with \eqref{tom}, recalling that
$P(\widetilde\Om_n({\delta_n}))= \mathcal H^1(\partial
\widetilde\Om_n({\delta_n}))$, we obtain
\begin{eqnarray}\nonumber
  E_\lambda(\Om_n)&\ge& P(\Om_n({\delta_n})) +\int_{\Om_n} g \,dx
                        -\omega^\delta_n  + \mathcal H^1(\partial
                        \widetilde\Om_n({\delta_n})) 
  \\ \label{trump}
                  && + 
                     \lambda \min_{t\in [0,1]}\left(
                     t^2\I_1(\Om^{\delta_n}) +
                     (1-t)^2\I_1(\widetilde\Omega_n({\delta_n}))\right) 
  \\
  \nonumber
                  &\ge& P(\Om_n({\delta_n})) +\int_{\Om_n} g\, dx
                        -\omega_n^\delta +  
                        \min_{t\in [0,1]}\left(\lambda
                        t^2\I_1(\Om^{\delta_n}) +
                        2\pi(1-t)\sqrt\lambda\right)
  \\
  \nonumber
                  &=&       P(\Om_n({\delta_n})) +\int_{\Om_n} g\, dx
                      -\omega_n^\delta +  \begin{cases}
                          2\pi\sqrt\lambda - \dfrac{\pi^2}{\I_1(\Om^{\delta_n})} & \text{if
                          } \I_1(\Om^{\delta_n}) > {\pi \over \sqrt{\lambda}} 
                          \\
                          \lambda\,\I_1(\Om^{\delta_n})& \text{if }
                          \I_1(\Om^{\delta_n}) \le {\pi \over
                            \sqrt{\lambda}}  
\end{cases}.                                    
\end{eqnarray}
Therefore, thanks to the lower semicontinuity of the perimeter
  with respect to the $L^1$ convergence (notice that
  $|\Omega_n(\delta_n)\Delta\Omega^{\overline{\delta}}|\to0$ as
  $n\to+\infty$) and thanks to the semicontinuity of $\I_1$ in Lemma
  \ref{1maggio}, in the limit as $n\to \infty$ we obtain
\begin{eqnarray}\label{ricerca}
  \liminf_{n\to \infty}E_\lambda(\Om_n)\ge P(\Om^{\bar \delta}) +
  \int_\Om g\, dx +  \begin{cases}
      2\pi\sqrt\lambda - \dfrac{\pi^2}{\I_1(\Om^{\bar \delta})} & \text{if
      } \I_1(\Om^{\bar \delta}) > {\pi \over \sqrt{\lambda}} 
      \\
      \lambda\,\I_1(\Om^{\bar \delta})& \text{if } \I_1(\Om^{\bar
        \delta}) \le {\pi \over \sqrt{\lambda}} 
\end{cases},
\end{eqnarray}

Letting now $\delta \to 0$ in \eqref{ricerca}, and again using
that $\I_1(\Om^{\bar \delta})\to\I_1(\Om)$ {by Lemma \ref{1maggio}},
we finally get
\begin{eqnarray}\nonumber
  \liminf_{n\to \infty}E_\lambda(\Om_n)&\ge& P(\Om) + \int_\Om g\, dx
                                             + \begin{cases} 
    2\pi\sqrt\lambda - \dfrac{\pi^2}{\I_1(\Om)} & \text{if } \lambda> \lambda_{\Omega}
    \\
    \lambda\,\I_1(\Om)& \text{if } \lambda \le  \lambda_{\Omega}
\end{cases}
\\\label{ricercatore}
&=&{\E_\lambda}(\Om),
\end{eqnarray}
{where $\E_\lambda(\Om)$ is defined in \eqref{Erelaxeddef}.}

\smallskip

We now have to show that there exists a sequence $\Om_n$ in $\S_m$
such that $|\Omega_n\Delta \Om|\to 0$ as $n\to +\infty$ and
\begin{equation}\label{minny}
  \limsup_{n\to \infty}E_\lambda(\Om_n)\le{\E_\lambda}(\Om).
\end{equation}
Recalling Corollary \ref{cor:schmidt}, it is enough to find a sequence
$\Om_n$ in $\K_m$ with the desired properties.

If $\lambda\le \lambda_{\Omega}$ we can take $\Om_n:=\Om$ and there is
nothing to prove.  If $\lambda>\lambda_\Om$ we let $R>0$ such that
$\Om\subset B_{R/2}(0)$. Notice that, for all $n$ large enough
(depending on $R$) there exist $n$ points $x_1,\ldots,x_n$ in
$B_{2R}(0)\setminus B_R(0)$ such that $|x_i-x_j|\ge R/\sqrt n$ for all
$i\ne j$.  We then take
$\Om_n:= \rho_n\,\Om \cup\left(\cup_{i=1}^n B_{r/n}(x_i)\right)$,
where
\begin{equation}\label{rho}
r := \frac{\sqrt\lambda - \sqrt{\lambda_\Om}}{2} \qquad \text{and} \qquad \rho_n:=\sqrt{1-\frac{\pi r^2}{m\,n}}\,.
\end{equation}
Notice that with th{ese choices of $r$ and} $\rho_n$ we have that the
sets $\rho_n \Omega$ and $B_{r/n}(x_i)$ are disjoint, $|\Om_n|=m$ and
\begin{align}
{\rm dist}(\rho_n\Om, \cup_{i=1}^n B_{r/n}(x_i))\ge \frac R2 - \frac rn\ge \frac R4
\end{align}
for all $n$ large enough.  Letting $t=\sqrt{\lambda_\Om/\lambda}$,
by Lemma \ref{lemdue} we estimate  
\begin{eqnarray}\nonumber
  E_\lambda(\Om_n)
  &\le& 
        E_{\lambda t^2}(\rho_n\Om) + E_{\lambda(1-t)^2}\Big(\cup_{i=1}^n B_{r/n}(x_i)\Big)
        + \frac{2t(1-t)\lambda}{{\rm dist}(\rho_n\Om, \cup_{i=1}^n B_{r/n}(x_i))}
  \\
  &\le& E_{\lambda t^2}(\rho_n\Om) + E_{\lambda(1-t)^2}\Big(\cup_{i=1}^n B_{r/n}(x_i)\Big)
        + \frac{2\lambda}{R} \label{Elballs} \\ \nonumber
  &=& E_{\lambda_\Omega}(\rho_n\Om) + E_{(\sqrt{\lambda}-\sqrt{\lambda_\Omega})^2}\Big(\cup_{i=1}^n B_{r/n}(x_i)\Big)
      + \frac{2\lambda}{R}.
\end{eqnarray}
Let now $\mu_i$ be the equilibrium measure for $B_{r/n}(x_i)$. Then $\frac1n \sum_{i=1}^n\mu_i$ 
is an admissible measure in the definition of $\I_1(\cup_i B_{r/n}(x_i))$, so that
\begin{align}
\begin{aligned}
\I_1(\cup_{i=1}^n B_{r/n}(x_i))&\le n\cdot\frac{1}{n^2}\I_1(B_{r/n}(x_1))+\frac{1}{n^2}\sum_{i\ne j}
\int_{B_{r/n}(x_i)}\int_{B_{r/n}(x_j)}\frac{d\mu_i(x)\,d\mu_j(y)}{|x-y|}\\
&\le \frac{1}{n}\I_1(B_{r/n}(x_1))+\frac{1}{n^2}\sum_{i\ne j}\frac{1}{|x_i-x_j|-\frac{2r}{n}}\\
&\le \frac{1}{n}\I_1(B_{r/n}(x_1))+\frac{2}{n^2}\sum_{i\ne j}\frac{1}{|x_i-x_j|},
\end{aligned}
\end{align}
for $n$ large enough. Since for any $i=1,\dots,n $ we have  
\begin{align}
  \int_{B_{r/n}(x_i)}g{(y)\, dy} \le {\frac{\pi
      r^2}{n^2}\|g\|_{L^\infty(B_{2R}(0))}},
\end{align}
from \eqref{Elballs} we obtain
\begin{equation}\label{eqeqeq}
\begin{aligned}
  E_\lambda(\Omega_n)&\le E_{\lambda_\Omega}(\rho_n\Omega)+ n
    P(B_{r/n}(x_1))
  +{\frac{(\sqrt{\lambda}-\sqrt{\lambda_\Omega})^2}{n}}\I_1(B_{r/n}(x_1))
  +\frac{\pi r^2}{n}\|g\|_{L^\infty(B_{2R}(0))}\\
  &\quad+\frac{2(\sqrt{\lambda}-\sqrt{\lambda_\Omega})^2}{n^2}\sum_{i\ne
    j}\frac{1}{|x_i-x_j|}+\frac{2\lambda}{R}\\
  & =
  E_{\lambda_\Omega}(\rho_n\Omega)+2\pi r
  +\frac\pi{2r}(\sqrt{\lambda}-\sqrt{\lambda_\Omega})^2
  +\frac{\pi r^2}{n}\|g\|_{L^\infty(B_{2R}(0))}\\
  &\quad+\frac{2(\sqrt{\lambda}-\sqrt{\lambda_\Omega})^2}{n^2}\sum_{i\ne
    j}\frac{1}{|x_i-x_j|} +\frac{2\lambda}{R}\\
    & = E_{\lambda_\Omega}(\rho_n\Omega)+2\pi (\sqrt{\lambda}-\sqrt{\lambda_\Omega})
  +\frac{\pi (\sqrt{\lambda}-\sqrt{\lambda_\Omega})^2}{4n}\|g\|_{L^\infty(B_{2R}(0))}\\
  &\quad+\frac{2(\sqrt{\lambda}-\sqrt{\lambda_\Omega})^2}{n^2}\sum_{i\ne
    j}\frac{1}{|x_i-x_j|} +\frac{2\lambda}{R},
\end{aligned}
\end{equation}
where we used \eqref{rho} and the fact
that $\I_1(B_r) = {\pi \over 2r}$ (see \cite[Equation
(2.5)]{MurNovRuf}).  Notice that, since $|x_i-x_j|\ge R/\sqrt n$, we have 
\begin{eqnarray*}
  \sum_{i\ne
  j}\frac{1}{|x_i-x_j|}
  \le  \frac {C n^2}{R}\,,
\end{eqnarray*}
for some universal constant $C>0$ and $n$ large enough depending only
on $R$.  Notice also that $\rho_n\to 1$ as $n\to \infty$, so that
\begin{align}
\lim_{n\to\infty}E_{\lambda_\Omega}(\rho_n\Omega) = \lim_{n \to
  \infty} \left( \rho_n P(\Omega) + \lambda_\Omega \rho_n^{-1}
  \mathcal I_1(\Omega) + \int_{\rho_n \Omega} g(x) dx \right) =
E_{\lambda_\Omega}(\Omega),
\end{align}
where in the last term we passed to the limit using $g \in \mathcal G$
and the Dominated Convergence Theorem. {Sending} $n\to\infty$ in
\eqref{eqeqeq} we then get
\begin{eqnarray*}
\limsup_{n\to\infty}E_\lambda(\Omega_n)
&\le& E_{\lambda_\Omega}(\Omega)+
2\pi (\sqrt{\lambda}-\sqrt{\lambda_\Omega})
+\frac{2\lambda + C(\sqrt{\lambda}-\sqrt{\lambda_\Omega})^2}{R}\,.
\end{eqnarray*}
{Sending} now $R\to +\infty$, we eventually obtain \eqref{minny}
and this concludes the proof.
\end{proof}

From Proposition \ref{teorelpro} we get the following result:

\begin{proposition}\label{correl}
  The functional $E_\lambda$ is lower semicontinuous in $\K_m$ if and
  only if $\lambda\le \lambda_c(m)$.
\end{proposition}

\begin{proof}
  Since $\I_1(\Om)\le \I_1(B_m) $ for any $\Om\in \K_m$ (see
  \cite[VII.7.3, p.157]{polya2}), where $B_m$ is a ball of measure
  $m$, we have that $\lambda_\Omega\ge \lambda_{B_m}=4m/\pi$, with
  equality if and only if $\Omega=B_m$. Thus, if $\lambda\le 4m/\pi$,
  the energy $E_\lambda$ coincides with its lower semicontinuous
  envelope $\overline{E_\lambda}$ by Proposition \ref{teorelpro}.  On
  the other hand, if $\lambda>\lambda_{B_m}$ then
  $\overline{E_{\lambda}}(B_m)<E_\lambda(B_m)$. Indeed, recalling
  \eqref{lambda} we have
  \begin{eqnarray*}
  \frac{E_\lambda(B_m)-\overline{E_{\lambda}}(B_m)}{\sqrt\lambda-\sqrt{\lambda_{B_m}}} &=&
  \left(\sqrt\lambda + \sqrt{\lambda_{B_m}}\right)\I_1(B_m)-2\pi \\
 &=& 
 \left(\sqrt\lambda + \frac{\pi}{\I_1(B_m)}\right)\I_1(B_m)-2\pi
 \\
 &=& \sqrt\lambda\,\I_1(B_m)-\pi =  \left(\sqrt\lambda - \sqrt{\lambda_{B_m}}\right)\I_1(B_m) >0.
  \end{eqnarray*}
  In particular, $E_\lambda$ is not lower semicontinuous for $\lambda>\lambda_{B_m}$.
\end{proof}

{Lastly,} Theorem \ref{teorel} directly follows from Proposition
\ref{teorelpro} and Proposition \ref{correl}.

\section{Existence of minimizers: Proof of Theorem \ref{teoexistence} }\label{teo2}

In this section we show existence of minimizers of $E_\lambda$ under
suitable assumptions on $\lambda$ and on the function $g$. We start
with a simple existence result for minimizers of the relaxed energy
$\overline{E}_\lambda$.

\begin{proposition}\label{existence}
  Let $g\in\mathcal G$. Then $\overline{E}_\lambda$
  admits a minimizer $\Omega_\lambda$ over $\mathcal A_m$ for every
  $\lambda > 0$.
\end{proposition}

\begin{proof}
  Let $\Omega_k$ be a minimizing sequence for
  $ \overline E_{\lambda}(\Omega) $.  Notice that $P(\Omega_k)<c$ for
  some positive constant $c$ independent of $k$. Letting
  $\Omega_k^R:=\Omega_k\cap B_{R}(0)$, we have
  $P(\Omega_k^R)\le P(\Omega_k)+P(B_R(0) )\le c+{2 \pi R}$. Thus, by
  the compactness of the immersion of $BV(B_R(0))$ into $L^1(B_R(0))$,
  applied to the sequence $\chi_{\Omega_{k}^R}$, we get that there
  exists a set $\Omega^R {\subset B_R(0)}$ such that
  $\chi_{\Omega_k^R}\to\chi_{\Omega^R}$ in $L^1$, up to a (not
  relabeled) subsequence, as $k\to +\infty$.  {Sending}
  $R\to +\infty$, by a diagonal argument we get that there exists
  $\Omega_\lambda \subset \R^2$ such that, up to extracting a further
  subsequence, the functions $\chi_{\Omega_k}$ converge to
  $\chi_{\Omega_\lambda}$ in $L^1_{\rm loc}(\R^{2})$.
  
    Now we observe that,
  since $\Omega_k$ is a minimizing sequence for $\overline E_\lambda$,
 there exists $C>0$ such that, for all
    $R > 0$ large enough, we have
    \begin{align}
  |\Omega_k\setminus B_R(0) |\inf_{{x \in B_R^c(0)} } g(x)\le
    \int_{\Omega_k\setminus B_R(0) }g(x)\,{dx} \le \int_{\Omega_k}g(x)\, {dx}\le C,
\end{align}
so that 
\begin{equation}\label{revolver}
  |\Omega_k\setminus B_R(0) |\le \frac{C}{\inf_{{x \in B_R^c(0)} }g(x)}\,. 
\end{equation}
In particular, by \eqref{coerc} 
for any $\eps>0$ there exists $R_\eps>0$ such that  $|\Omega_k\setminus B_{R_\eps}(0)|\le\eps$ for all $k$.
Thus, recalling the convergence of $\chi_{\Omega_k}$ to $\chi_{\Omega_\lambda}$ in $L^1_{\rm loc}(\R^{2})$ as $k\to \infty$,
there also exists $k_\eps\in\mathbb N$ such that 
\begin{align}
  |\Omega_k\Delta\Omega_\lambda|=|(\Omega_k\Delta\Omega_\lambda)\cap B_{R_\eps}(0)|
  +|(\Omega_k\Delta\Omega_\lambda)\setminus B_{R_\eps}(0)|\le 2\eps,
\end{align}
for all $k\ge k_\eps$, that is, the sequence $\chi_{\Omega_k}$ converges to 
$\chi_{\Omega_\lambda}$ in $L^1(\R^{2})$ as $k\to \infty$.

Since, by definition, $\overline E_\lambda$ is
lower semicontinuous in $L^1(\R^{2})$, we eventually get that $\Omega_\lambda$ is a
minimizer of $\overline E_\lambda$.
\end{proof}

The main difficulty in proving Theorem \ref{teoexistence} is to show
that the minimizer $\Omega_\lambda$ is indeed an element of $\K_m$, so
that it is also a minimizer of $E_\lambda$ by Proposition
\ref{teorelpro}.

We first show that ${\rm cap}_1(\Omega)$ depends continuously on
smooth perturbations of $\Omega$, where $\Om\subset\R^2$ is a compact
set with Lipschitz boundary.

\begin{lemma}\label{l4missing}
  Let $\Om\subset\R^2$ be a compact set with positive measure and Lipschitz boundary.
  Let $\eta\in W^{1,\infty}(\R^2, \R^2)$, let
  $\Phi_t(x):=x+t\eta(x)$ be the corresponding family of (Lipschitz) diffeomorphims, defined for $t\in (-t_0,t_0)$
  and $t_0$ sufficiently small,
  and let $\Om_t:=\Phi_t(\Om)$. 
  
  Then, for $t\in (-t_0,t_0)$ there holds
  \begin{eqnarray}\label{capdue}
   {\rm  cap}_1(\Omega)\le {\rm  cap}_1(\Omega_t) (1+Ct),
  \end{eqnarray}
 where the constant $C>0$ depends only on the $W^{1,\infty}$-norm of $\eta$.
\end{lemma}

\begin{proof}
  Let $u_t$ be the $\frac 12$-capacitary potential of $\Om_t$
  minimizing \eqref{capcap} with $\Om$ replaced by $\Om_t$, and let
  $u:=u_t\circ \Phi_t$. Notice that $u$ is an admissible function for
  the minimum problem \eqref{capcap}. In particular, we have
\begin{equation}\label{tt}
{\rm  cap}_1(\Omega)\le \frac{1}{4\pi}\int_{\R^2}\int_{\R^2}\frac{|u(x)-u(y)|^2}{|x-y|^3}\,dxdy\,.
\end{equation}
We now compute
\begin{eqnarray}\label{tata}
\int_{\R^2}\int_{\R^2}\frac{|u(x)-u(y)|^2}{|x-y|^3}\,dxdy &=& 
\int_{\R^2}\int_{\R^2}\frac{|u_t(\Phi_t(x))-u_t(\Phi_t(y))|^2}{|x-y|^3}\,dxdy
\\ \nonumber
&=& \int_{\R^2}\int_{\R^2}\frac{|u_t(X)-u_t(Y)|^2}{|\Phi_t^{-1}(X)-\Phi_t^{-1}(Y)|^3}
\,\left|{\rm det}\nabla\Phi_t^{-1}(X)\right| \left|{\rm det}\nabla\Phi_t^{-1}(Y)\right|\,dXdY,
\end{eqnarray} 
where we performed the change of variables $X=\Phi_t(x)$, $Y=\Phi_t(y)$.
Observing that 
\begin{align}
|{\rm det}\nabla\Phi_t^{-1}(X) - 1| \le Ct\qquad \text{and}\qquad 
|\Phi_t^{-1}(X)-\Phi_t^{-1}(Y)| \ge (1-Ct)|X-Y|, 
\end{align}
where $C>0$ depends only on the $W^{1,\infty}$-norm of $\eta$,
from \eqref{tt} and \eqref{tata} we readily obtain \eqref{capdue}.
\end{proof}

From Lemma \ref{l4missing} and \eqref{capuno}, we immediately get the following result:

\begin{corollary}\label{cor4missing}
Under the assumptions of Lemma \ref{l4missing}, there holds
  \begin{equation}\label{capiuno}
  \I_1(\Omega_t)\le \I_1(\Omega) (1+Ct),
  \end{equation}
   where the constant $C>0$ depends only on the $W^{1,\infty}$-norm of $\eta$.
\end{corollary}

We now show that, if $\lambda<\lambda_c(m)$, we can decrease the
energy of a set $\Om\in \K_m$ by reducing the number of its connected
components and holes.

\begin{proposition}\label{probounded}
  Let $\lambda<\lambda_c(m)$ and $g\in\mathcal G$. Then, for any
  $\Om\in \K_m$ we can find $\widetilde\Omega\in \K_m$ such that
  $E_\lambda(\widetilde\Omega)\le E_\lambda(\Omega)$,
  $\widetilde\Omega\subset B_R(0)$, and the numbers of connected
  components of both $\widetilde \Omega$ and of $\widetilde \Omega^c$
  are bounded above by $N$, where the numbers $R,\,N$ depend only on
  $\lambda, m, g$ and $E_\lambda(\Omega)$.
\end{proposition}

\begin{proof}
We divide the proof into two steps.

\smallskip

\noindent\emph{Step 1: Construction of a uniformly bounded set with a uniformly bounded number of connected components.}

\noindent Let $\Omega_i$ be the connected components of $\Omega$, and
up to a relabeling we can suppose that if $m_{i}:=|\Omega_{i}|$, then
$m_{i}\ge m_{i+1}$.
Let $\eps\in (0,m/2)$. We claim that there
exists $N_\eps\in\mathbb N$ depending only on $\eps$ and $m$ such that
\begin{equation}\label{esclaim}
  \left| \Omega\setminus \bigcup_{i>N_\eps} \Omega_i \right| \ge
  m-\frac\eps 2>\frac 34 m. 
\end{equation}
Indeed, we have ${\sum_{i=1}^\infty} m_{i}= m$,
and by the isoperimetric inequality we get
\begin{align}
  {\sum_{i=1}^\infty} \sqrt{{4 \pi} m_{i}} \le \sum_{i=1}^\infty
  P(\Omega_i)\le E_\lambda(\Om).
\end{align}
Recalling that the sequence $i\mapsto m_{i}$ is decreasing, it
follows that 
\begin{align}
  m_{i}\le \frac{E_\lambda ^2 (\Om)}{4\pi i^2}.
\end{align}
Hence there exists $C>0$ depending only on $m, g$ and $E_\lambda(\Om)$ such that
\begin{align}
\sum_{i\ge k} m_{i}\le\frac{C}{k},
\end{align}
so that \eqref{esclaim} holds for $N_\eps\ge 2C/\eps$.

Let us set
\begin{align}
U_\eps:=\bigcup_{i=1}^{N_\eps} \Omega_i.
\end{align}
We claim that there exists $\bar R\ge 1$, depending only on $m, g$ and $E_\lambda(\Om)$ such that 
\begin{equation}\label{A}
|U_\eps \cap B_{\bar R}(0)|\ge\frac 23 m.
\end{equation}
Notice that the previous equation implies in particular that 
\begin{equation}\label{B}
|U_\eps \setminus B_{\bar R}(0)|\le\frac 13 m.
\end{equation}  
Indeed, reasoning as in the proof of  Proposition \ref{existence}, for any $R>0$ we can write
\begin{align}
E_\lambda(\Om)\ge\int_{\Omega\setminus B_R(0) }g\, {dx} \ge |\Omega\setminus
B_R(0) | \,  {\inf_{{x \in B_R^c(0)} }g(x)}, 
\end{align}
so that  
\begin{align}
  |\Omega\setminus B_R(0) |\le \frac{E_\lambda(\Om)}{{\inf_{{x \in
          B_R^c(0)} }g(x)}}.
\end{align}
Take now $\bar R\ge 1$ such that
$\frac{E_\lambda(\Om)}{{\inf_{{x \in B_{\bar R}^c(0)} }g(x)}}\le\frac{m}{12}$.
Such a radius exists in view of
the coercivity of $g$.  Then we have
\begin{align}
  |U_\eps\cap B_{\bar R}(0)|\ge |U_\eps| - \frac{m}{12} >
  \frac{3}{4}m- \frac{m}{12} = \frac{2}{3}m,
\end{align}
which gives \eqref{A}.

By the same argument, there exists $R_\eps\ge 2\bar R$ such that
$|U_\eps \cap B_{R_\eps}(0)|\ge m-\eps$. Moreover, since
$P(U_\eps)\le E_\lambda(\Om)$, we can also find a radius
$R_\eps^n\in [R_\eps,R_\eps']$, with
  $R_\eps' := R_\eps+2E_\lambda(\Om)$, such that
$\mathcal H^1(U_\eps \cap \partial B_{R_\eps^n}(0))=0$.

Let now $\varphi:\R\to\R$ be {a cutoff function} defined as
\begin{align}
\varphi(s) := \left\{\begin{array}{ll}
1 & \text{if }|s|\le \bar R
\\
2 - \dfrac{|s|}{\bar R} & \text{if }\bar R<|s|\le 2\bar R
\\
0 & \text{if }|s|\ge 2\bar R.
\end{array}\right.
\end{align}
For $t\ge 0$, we let $\Phi_t(x):= (1 + t\varphi(|x|)) x$ and we
notice that 
\begin{align}
{\rm det}\,\nabla \Phi_t(x) = (1 + t\varphi(|x|))^2 + 
t|x|\varphi'(|x|) {+ t^2 |x| \varphi(|x|) \varphi'(|x|)} \,.
\end{align} 
In particular, the map
  $t\mapsto |\Phi_t(A)| = \int_A {\rm det}\,\nabla \Phi_t(x) \,dx $ is
  continuous for every set $A \subset \R^2$ of finite
  measure. Recalling \eqref{A}, \eqref{B} and letting
$\widetilde U_\eps := U_\eps\cap B_{R^n_\eps}(0)$, we have 
\begin{eqnarray*}
  |\Phi_t(\widetilde U_\eps)|
  &=& \int_{\widetilde U_\eps} {\rm
      det}\,\nabla \Phi_t(x) \,dx \,
      = \int_{\widetilde U_\eps} \left[(1 + t\varphi(|x|))^2 +
      t|x|\varphi'(|x|)+ t^2 |x| \varphi(|x|) \varphi'(|x|)\right]\,dx  
  \\
  &\geq&
              |\widetilde U_\eps| + (2t + t^2) |U_\eps\cap B_{\bar R}(0)|
              - \frac{t+t^2}{\bar R} \int_{U_\eps\cap B_{2\bar R}(0)\setminus B_{\bar R}(0)} |x|\,dx
  \\
  &\ge& |\widetilde U_\eps| + \frac 23 m (2t+t^2) - \frac 23 m
        (t+t^2)\, =\, |\widetilde U_\eps|  + \frac 23 m t\,, 
\end{eqnarray*}
which implies that 
\begin{align}
\left|\Phi_\frac{3(m-|\widetilde U_\eps|)}{2m}(\widetilde U_\eps)\right|\ge m\,.
\end{align}
Noting that
$|\Phi_0(\widetilde U_\eps)|= |\widetilde U_\eps|= |U_\eps\cap
B_{R^n_\eps}(0)|\le m$, we obtain that there exists $t_\eps\ge 0$ such
that $|\Phi_{t_\eps}(\widetilde U_\eps)| = m$ and
\begin{equation}\label{teps}
t_\eps\le \frac{3(m-|\widetilde U_\eps|)}{2m}\le \frac{3\eps}{2m}\,.
\end{equation}

Let now $W_\eps:=\Phi_{t_\eps}(\widetilde U_\eps)$. Recalling
Corollary \ref{cor4missing} and \cite[Proposition 3.1]{hof} (see also
\cite[Proposition 17.1]{maggi}), the following properties hold:
\begin{itemize}
\item[$(i)$] $W_\eps\subset B_{R_\eps'}(0)$ and $W_\eps$ has at most
  $N_\eps $ connected components;
\item[$(ii)$] $|W_\eps|=m$; 
\item[$(iii)$] $P(W_\eps)\le {\rm Lip}(\Phi_{t_\eps})\, P(\widetilde U_\eps)\le 
(1+t_\eps)P(\widetilde U_\eps)\le P(\widetilde U_\eps)+Ct_\eps$;
\item[$(iv)$] $\I_1(W_\eps)\le \I_1(\widetilde U_\eps)+Ct_\eps$;
\item[$(v)$] $\int_{W_\eps} g(x)\le \int_{\widetilde U_\eps} g(x)\,dx + C t_\eps$;
\end{itemize}
where the constant $C>0$ depends only on $g$, $m$ and $E_\lambda(\Om)$.
Indeed, the first two assertions follow by construction. Assertion
$(iii)$ holds true since $\|\varphi\|_{L^\infty(\R)}\le1$. 
Assertion $(iv)$ follows by Corollary \ref{cor4missing}, while $(v)$
holds true since 
\begin{eqnarray*}
  \int_{W_\eps} g(x)\,dx
  &=& \int_{\widetilde U_\eps}
      g(\Phi_{t_\eps}(x)) {\rm det}\,\nabla
      \Phi_{t_\eps}(x)  \,dx
  \\
  &\le& (1+t_\eps)^2 \left( \int_{\widetilde U_\eps} g(x)\,dx +
        C \|\nabla g\|_{L^\infty(B_{2\bar R}(0))} m t_\eps\right) 
  \\ &\le& \int_{\widetilde U_\eps} g(x)\,dx + C'  t_\eps\,,
\end{eqnarray*}
for some $C, C'>0$ depending only on $g$, $m$ and  $E_\lambda(\Om)$.

We claim that $E_\lambda(W_\eps)\le E_\lambda(\Omega)$  for $\varepsilon$ small enough.  
Letting $V_\eps := \Omega\setminus \widetilde U_\eps$, we compute
\begin{eqnarray}\label{ilcalcolazzo}
  \delta_\eps
  &:=& E_\lambda(\Om) -E_\lambda(W_\eps) 
       =
       P(\widetilde U_\eps)+P(V_\eps)+\lambda\I_1(\widetilde
       U_\eps\cup V_\eps)+\int_{\widetilde U_\eps}g\, {dx} 
       +\int_{V_\eps}g\, {dx}
       \notag \\
  && -P(W_\eps)-\lambda\I_1(W_\eps)-\int_{W_\eps}g\, dx
     \notag \\
  &\ge& P(V_\eps)+\lambda\I_1(\widetilde U_\eps\cup
        V_\eps)-\lambda \I_1(\widetilde U_\eps)-C t_\eps\,, 
\end{eqnarray}
where the constant $C>0$ depends only on $g$, $m$, and
$E_\lambda(\Omega)$. We also have $|V_\eps| < \eps$ by
  construction.  Using Lemma \ref{distributionofcharges} with
$U=\widetilde U_\eps$ and $V=V_\eps$, we obtain
  \begin{equation}\label{calc4}
  P(V_\eps)+\lambda\I_1(\widetilde U_\eps\cup V_\eps)-\lambda I_1(\widetilde U_\eps)
  \ge P(V_\eps)\left(1-\frac{\lambda\pi}{{4|\widetilde U_\eps|}}\right)
  \ge P(V_\eps)\left(1-\frac{\lambda\pi}{{4(m-\eps)}}\right).
  \end{equation}
  Recalling that $\lambda<4m/\pi$, we can choose $\eps$ small enough so that
  \begin{align}
  \left(1-\frac{\lambda\pi}{{4(m-\eps)}}\right)\ge \frac 12 \left(1-\frac{\lambda\pi}{{4m}}\right).
  \end{align}
  Recalling also \eqref{teps}, with the help of the isoperimetric
    inequality we then get
  \begin{align}
  \delta_\eps\ge \frac 12 \left(1-\frac{\lambda\pi}{{4m}}\right) P(V_\eps)
  - \frac{3C}{2m}\,|V_\eps|
  \ge \sqrt\pi \left(1-\frac{\lambda\pi}{{4m}}\right) |V_\eps|^\frac 12 - \frac{3C}{2m}\,|V_\eps| \ge 0,
  \end{align}
  provided we choose $\eps$ small enough depending only on $g$, $m$
  and $E_\lambda(\Om)$.  We thus proved that $\delta_\eps\ge 0$, that
  is, $E_\lambda(W_\eps)\le E_\lambda(\Omega)$.

\smallskip

\noindent\emph{Step 2: Construction of a set with a uniformly bounded number of holes.}

\noindent In Step 1 we built a set $W\in\K_m$, with a uniformly
bounded number of connected components and such that
$E_\lambda(W)\le E_\lambda(\Omega)$.  In particular, there exists a
uniform radius $R>0$ such that $W\subset B_R(0)$.  Starting from this,
we construct another set with a uniformly bounded number of holes,
where a hole is a bounded connected component of the complement set.

Let us denote by $\{H_{i}\}_{i\in\mathbb N}$ the connected components
of $W^c$ which are bounded.  As in Step 1, for $\eps\in (0,m/2)$ we
can find $N_\eps$ such that $ \sum_{i> N_\eps}|H_{i}|\le \eps.$ Let us
set $H_\eps:=\bigcup_{i>N_\eps}H_{i}$ and
\begin{align}
  \Omega_\eps:= \sqrt\frac{m}{m+|H_\eps|}\, \left(W\cup
    H_\eps\right)\,\in\,\K_m.
\end{align}
Notice that 
\begin{eqnarray}\label{eqP}
  P(\Om_\eps) &\le& P(W\cup H_\eps) = P(W) - P(H_\eps),
  \\ \label{eqI}
  \I_1(\Om_\eps) &=& \sqrt\frac{m+|H_\eps|}{m}\,\I_1(W\cup H_\eps)\le
                     {\left( 1 +\frac{|H_\eps|}{2m} \right)} \,
                     \I_1(W), 
  \\ \label{eqgg}
  \int_{\Om_\eps} g(x)\,dx &=& \frac{m}{m+|H_\eps|}\int_W g\left(
                               \sqrt\frac{m}{m+|H_\eps|}\,x\right)\,dx \nonumber 
  \\ &\le& \int_{W} g(x)\,dx + C |H_\eps|,
\end{eqnarray}
where the constant $C > 0$ depends on $m,\,g$ and $R$, and in
  obtaining \eqref{eqI}, we used concavity of the square root and
  monotonicity of the capacitary term with respect to filling the
  holes.  Putting together \eqref{eqP}, \eqref{eqI} and \eqref{eqgg},
we then get
\begin{align}
  E_\lambda(\Omega_\eps)\le E_\lambda(W) - P(H_\eps) +
  \left(\frac{E_\lambda(\Omega)}{2m} + C\right) |H_\eps|\,,
\end{align}
which yields the claim for $\eps$ small enough by the isoperimetric
inequality.
\end{proof}

We now prove Theorem \ref{teoexistence}.

\begin{proof}[Proof of Theorem \ref{teoexistence}]
  Let $\Omega_n\in\K_m$ be a minimizing sequence for $E_\lambda$. In
  particular $E_\lambda(\Omega_n)\le c$, for some
  $c = c(\lambda,g,m)>0$ depending only on $g$ and $m$.

Thanks to Proposition \ref{probounded}, we can assume that the sets
$\Omega_n$ are uniformly bounded and the number of connected
components both of $\Omega_n$ and of $(\Omega_n)^c$ is uniformly
bounded.  In particular, the number of connected components of
$\partial{\Omega}_n$ is also uniformly bounded.

Since $\mathcal H^1(\partial{\Omega}_n)\le c$, it follows by Blaschke Theorem (see \cite[Theorem 4.4.15]{AmTi}) that 
$\partial{\Omega}_n\to\Gamma$ in the Hausdorff distance, as $n\to +\infty$ up to passing to a subsequence, 
for some compact set $\Gamma\subset\R^2$ with $\mathcal H^1(\Gamma)<+\infty$.  

Up to passing to a further subsequence, we also have that the sets 
${\Omega}_n$ converge to some compact set  $\Omega$, again in the Hausdorff distance.   
We notice that 
\begin{equation}\label{inside}
\partial\Omega\subset\Gamma.
\end{equation}
Indeed if $x\in\partial\Omega\setminus\Gamma$, then there exists $x_n\in \Omega_n$ such that $x_n\to x$. On the other hand, there exists $N\in\mathbb N$ and $\varepsilon_0>0$ such that $d_H(x_n,\partial\Omega_n)\ge\varepsilon_0$ for $n\ge N$. Otherwise there would exists $y_n\in\partial\Omega_n$ such that $|y_n-x_n|=d(x_n,\partial \Omega_n)\to 0$ and thus 
\begin{align}
|y_n-x|\le |y_n-x_n|+|x_n-x|\to0,
\end{align}
which is impossible, since $x\notin\Gamma$. But then the ball
$B_{\varepsilon_0}(x_n)$ is contained, for $n\ge N$, in $\Omega_n$ and
converges in Hausdorff distance to
$B_{\varepsilon_0}(x)\subset\Omega$. In particular we get
$x\notin\partial\Omega$, which gives a contradiction.  Thanks to Golab
Theorem \cite[Theorem 4.4.17]{AmTi}, we then obtain
\begin{equation}\label{sc}
\mathcal H^1(\partial\Omega)\le \mathcal H^1(\Gamma)\le \liminf_{n\to+\infty} \mathcal H^1(\partial\Omega_n)\le c.
\end{equation}

Let now
$x\in\mathbb R^N\setminus\Gamma$. Then there exist $\varepsilon>0$
and $N\in\mathbb N$ such that for $n\ge N$, we have that
$B_\varepsilon(x)\subset \Omega_n$ or
$B_\varepsilon(x)\subset (\mathbb R^N\setminus\Omega_n)$. Thus
$\chi_{\Omega_n}(x)=1$ or $\chi_{\Omega_n}(x)=0$ for $n$ large enough.
In particular $\chi_{\Omega_n}\to \chi_\Omega$ almost everywhere 
and, by the Dominated Convergence Theorem, we
obtain that $\chi_{\Omega_n}\to \chi_\Omega$ in $L^1(\R^2)$.  

We can now conclude that $\Omega$ is a minimizer in $\K_m$. Indeed,
the minimality is granted by the lower semicontinuity of $E_\lambda$
w.r.t. the $L^1$-convergence, since
$\lambda<\lambda_c(m)$. Moreover, $\Omega\in \K_m$ since it is
compact, it has measure $m$ and $\mathcal H^1(\partial\Omega)<+\infty$
by \eqref{sc}.  The proof is concluded.
\end{proof}

\section{Partial regularity of minimizers: Proof of Theorem \ref{teodensity}}\label{teo3}

In this section we show that all minimizers of our problem are
bounded, contain finitely many connected components and holes, and
satisfy suitable density estimates. For the latter, the argument
essentially follows the classical proof of the density estimates for
quasi-minimizers of the perimeter (see \cite{maggi}). Note, however,
that the standard regularity theory of quasi-minimizers of the
perimeter cannot be applied directly, as the nonlocal term $\I_1$
presents a crititcal perturbation to the perimeter. This additional
complication may be overcome with the help of Lemma
\ref{distributionofcharges} for subcritical values of
$\lambda < \lambda_c(m)$, yielding partial regularity of the
minimizers.

\begin{proof}[Proof of Theorem \ref{teodensity}] Throughout the
    proof, we identify the minimizer $\Omega_\lambda$ with its regular
    representative $\Omega_\lambda^+$, and drop the superscript ``+''
    for ease of notation.

  First of all, the assertion about the number of connected components
  of $\Omega_\lambda$ and $\Omega_\lambda^c$ follows from the argument
  in the proof of Proposition \ref{probounded}, observing that the
  inequality in that proposition becomes strict otherwise,
  contradicting the minimality of $\Omega_\lambda$. Therefore,
    the rest of the proof focuses on the density estimates
  \eqref{eqdens}, whose proof uses the estimates similar to those
  in the proof of Proposition \ref{probounded}. It is enough to show
  the first assertion, since the second one can be proved
  analogously, taking into account that $\Omega_\lambda^+$ is a
    closed set.

  For $r\in (0,\sqrt{m/(2\pi)})$, so that
  $|\Omega_\lambda\setminus B_r(x)|\ge m/2$, we set
\begin{align}
  v(0):=0,\qquad v(r):=|\Omega_\lambda\cap {B_r(x)}|,\qquad
  \Omega_{\lambda,r}:=
  \sqrt\frac{m}{m-v(r)}\,\big(\Omega_\lambda\setminus
  B_r(x)\big)\in\K_m.
\end{align}
Since $x \in \partial \Omega^+_\lambda$, we have that $v(r)>0$ for all
$r>0$.  Moreover, since $\Omega_\lambda$ has finite perimeter, for
almost every $r>0$ there holds (see \cite{maggi})
\begin{align}
P(\Omega_\lambda)=P(\Omega_\lambda;B_r(x))+P(\Omega_\lambda;B_r ^c (x))\qquad \text{ and }
\qquad \frac{dv}{dr}(r)=\mathcal H^1(\partial B_r(x) \cap\Omega_\lambda)\,.
\end{align} 
Recalling the Lipschitz continuity of $g$, for almost every $r\in (0,\sqrt{m/(2\pi)})$ we then get
\begin{equation}\label{form}
\begin{aligned}
  &P(\Omega_\lambda;B_r(x))+P(\Omega_\lambda;B_r ^c
  (x))+\lambda\I_1(\Omega_\lambda)+\int_{\Omega_\lambda}g(y)\, {dy} =
  E_\lambda(\Om_\lambda)
  \\
  &\le E_\lambda(\Om_{\lambda,r}) = \sqrt\frac{m}{m-v(r)}\,
  \left(P(\Omega_\lambda;B_r ^c (x))+\mathcal H^1(\partial B_r(x)
    \cap\Omega_\lambda)\right)  \\
  &+\lambda\sqrt\frac{m-v(r)}m\, \I_1(\Omega_\lambda\setminus
  B_r(x))+\frac{m}{m-v(r)}\,\int_{\Omega_\lambda\setminus
    B_r(x)}g\left(\sqrt\frac{m}{m-v(r)}\,y\right)\, dy
  \\
  &\le P(\Omega_\lambda;B_r ^c (x))+ \frac{dv}{dr}(r) +\lambda
  \I_1(\Omega_\lambda\setminus B_r(x)) + \int_{\Omega_\lambda}g(y)\,
  {dy} + Cv(r),
\end{aligned}
\end{equation}
where the constant $C>0$ depends only on $m$, $\lambda$ and $g$. 

After some simplifications, \eqref{form} reads
\begin{align}
  P(\Omega_\lambda;B_r(x))\le \frac{dv}{dr}(r)+Cv(r)+\lambda \I_1(\Omega_\lambda\setminus B_r(x)) -\lambda \I_1(\Om_\lambda).
\end{align}
Applying Lemma \ref{distributionofcharges} with $U=\Omega_\lambda\setminus  B_r(x)$ and $V=\Omega_\lambda\cap B_r(x) $, 
we then obtain
\begin{eqnarray}\nonumber
  P(\Omega_\lambda; B_r(x))&\le& \frac{dv}{dr}(r)
  +Cv(r)+\frac{\lambda\pi}{4|\Omega_\lambda\setminus  B_r(x)|}P(\Omega_\lambda\cap B_r(x) )
  \\ \label{covid}
  &\le& \frac{dv}{dr}(r)+C'v(r)+\frac{\lambda\pi}{4m}P(\Omega_\lambda\cap B_r(x) ),
\end{eqnarray}
where $C'>0$ depends only on $m$, $\lambda$ and $g$. 

Since for almost every $r>0$ there holds
\begin{align}
P(\Omega_\lambda; B_r(x))+\frac{dv}{dr}(r)=P(\Omega_\lambda\cap B_r(x)),
\end{align}
by adding the quantity $\frac{dv}{dr}(r)$ to both sides of \eqref{covid} we obtain
\begin{align}
P(\Omega_\lambda\cap B_r(x) )\left(1-  \frac{\lambda\pi}{4m}\right)\le 2\, \frac{dv}{dr}(r)+C'v(r).
\end{align}
Thanks to the isoperimetric inequality, for almost every $r\in (0,\sqrt{m/(2\pi)})$
we then get
\begin{align}
2\sqrt\pi\left(1-  \frac{\lambda\pi}{4m}\right) \sqrt{v(r)} \le 2\,\frac{dv}{dr}(r)+C'v(r).
\end{align}
Recalling that $\lambda<4m/\pi$, there exists $r_0\in (0,\sqrt{m/(2\pi)})$, depending only on $m$, $\lambda$ and $g$,
such that 
\begin{align}
2\sqrt\pi\left(1-  \frac{\lambda\pi}{4m}\right) \sqrt{v(r)}-C'v(r)\ge \sqrt\pi\left(1-  \frac{\lambda\pi}{4m}\right) \sqrt{v(r)}
\qquad \text{for all $0<r\le r_0$,}
\end{align}
which gives
\begin{align}
\frac{dv}{dr}(r)\ge \frac{\sqrt\pi}{2}\,\left(1-  \frac{\lambda\pi}{4m}\right) \sqrt{v(r)}\qquad \text{for a.e. $0<r\le r_0$.}
\end{align}
After a direct integration, this inequality implies that 
\begin{align}
  v(r)\ge \frac{\pi}{16}\left(1- \frac{\lambda\pi}{4m} \right)^2 r^2
  \qquad \text{for a.e. $0\le r\le r_0$,}
\end{align}
which gives the first inequality in \eqref{eqdens}.  This
  concludes the proof. 
\end{proof}

\section{Asymptotic shape of minimizers: Proof of Theorem
  \ref{teoasy}}\label{tasy}

\begin{proof}[Proof of Theorem \ref{teoasy}]

  The first assertion is a direct consequence of {Theorem
    \ref{teoexistence}, since $\lambda_k < \lambda_c(m_k)$ for $k$
    large enough}.
  
\smallskip

We now prove the second assertion. Let $\Omega_k$ be a minimizer of
$E_{\lambda_{k}}$ over $\K_{m_k}$.  Recalling Remark \ref{remamb},
without loss of generality we can assume that
\begin{align}
P(\Om_k)=\mathcal H^1(\partial \Om_k)\,.
\end{align}
By a change of variables $x = r_k \tilde x$, with
  $r_k := \sqrt{m_k / \pi}$ we obtain that
\begin{align}
  E_{\lambda_k}(\Omega_k)=r_k F_k(\widetilde \Omega_k),
\end{align}
where $\widetilde\Omega_k:= r_k^{-1} \Omega_k$, so that in
particular $|\widetilde\Omega_k|=\pi$, and
\begin{align}
  F_k(\Omega):=P(\Omega)+\frac{\lambda_k \pi}{m_k}\,\I_1(\Omega)+
  r_k \int_{\Omega} g\left(r_k  \tilde x\right)\,d \tilde x.
\end{align}

Observe that since $g \in \mathcal G$, there exists $x_0 \in \R^2$
  such that $g(x_0) = \min g$, and without loss of generality we may
  assume that $x_0 = 0$.
  By the minimality of $\Omega_k$ we have that
\begin{eqnarray}\label{eqin}
  P(\widetilde{\Omega}_k)+
  \frac{\lambda_k\pi}{m_k}\,\I_1(\widetilde{\Omega}_k)+  
  r_k \int_{\widetilde{\Omega}_k}
  g\left(r_k  \tilde x\right)\,d\tilde x 
  \le
  P(B_1(0))+\frac{\lambda_k\pi}{m_k}\,
  \I_1(B_1(0))   
  +r_k \int_{B_1(0)}
  g\left(r_k  \tilde x\right)\,d\tilde x. 
\end{eqnarray}
Notice also that, since the gradient of $g$ is locally bounded, we have
\begin{align}
  \label{eq:gg0}
  0 \leq \int_{B_1(0)} \left( g\left(r_k 
  \tilde x\right) - g(0) \right) \,d\tilde x \leq C r_k ,
\end{align}
where $C > 0$ depends only on $g$, for all $k$ large enough.
From \eqref{eqin} and \eqref{eq:gg0} we then get
\begin{equation}\label{cirillo}
  P(\widetilde{\Omega}_k)+\frac{\lambda_k\pi}{m_k}\,
  \I_1(\widetilde{\Omega}_k) +
  r_k\int_{\widetilde \Omega_k} \left( 
      g\left(r_k \tilde x\right) - g(0) \right) \,d\tilde x \le 
  P(B_1(0))+\frac{\lambda_k\pi}{m_k}\,\I_1(B_1(0))+ C r_k^2,
\end{equation}
and we note that the integral in the left-hand side is non-negative.

We recall from Lemma \ref{lemtom} the inequality 
\begin{align}
P(B_1(0))+\lambda\I_1(B_1(0))\le \mathcal H^1(\widetilde \Omega_{k})+\lambda \I_1(\widetilde \Omega_{k})=P(\widetilde \Omega_{k})+\lambda \I_1(\widetilde \Omega_{k}),
\end{align}
where the last equality follows from Remark \ref{remamb}, for all
$\lambda\le 4$. Hence we get
\begin{equation}\label{anne}
\I_1(B_1(0))-\I_1(\widetilde \Omega_{k})\le \frac 1\lambda \left(P(\widetilde{\Omega}_{k})-P(B_1(0))\right).
\end{equation}
From \eqref{cirillo} and \eqref{anne} with $\lambda=4$ we then obtain
\begin{align}
    \label{eq:Pgmk}
  \left(1-\frac{\lambda_k\pi}{4m_k}\right)
  \left(P(\widetilde{\Omega}_{k})-P(B_1(0))\right) +
  r_k \int_{\widetilde \Omega_k} \left(
  g\left(r_k \tilde x\right) - g(0) \right) \,d\tilde x
  \leq C r_k^2\,.
\end{align}

By the isoperimetric inequality in quantitative form \cite{fumapa},
there exist $\tilde x_k\in \R^2$ such that
\begin{align}
  \label{eq:isodef}
  |\widetilde{\Omega}_{k}\Delta B_1(\tilde x_k)|^2\le C
  m_k\left(1-\frac{\lambda_k\pi}{4m_k}\right)^{-1},
\end{align} 
for all $k$ small  enough,
for some constant $C > 0$ depending only on $g$. Hence, recalling the assumption on $\lambda_k,\,m_k$, we
obtain
\begin{align}
  \lim_{k\to +\infty}|\widetilde{\Omega}_{k}\Delta B_1(\tilde x_{k})| =0,
\end{align}
implying that $\widetilde \Om_k$ converge to $B_1(0)$ is the
$L^1$-sense. Hausdorff convergence then follows from the fact that the
density estimates in Theorem \ref{teodensity} can be easily seen to
hold for $\widetilde \Om_k$ uniformly in $k$.

Similarly, from \eqref{eq:Pgmk} written in the original unscaled
  variables, and with the help of the isoperimetric inequality we
  infer that
  \begin{align}
    \label{eq:mchik} 
    {1 \over m_k} \int_{\R^2} \chi_{\Omega_k}(x) \bar g(x) \, dx - g(0)
    \leq C m_k^{1/2},   
  \end{align}
  where $\bar g(x) = \min \{g(x), g(0) + 1\}$ and $\chi_{\Omega_k}$
  are the characteristic functions of $\Omega_k$. At the same time,
  defining $x_k := r_k \tilde x_k$ and using \eqref{eq:isodef} we have
  \begin{align}
    \left| \int_{\R^2} (\chi_{\Omega_k} - \chi_{B_{r_k}(x_k)} )
      \bar  g \, dx \right| \leq (g(0) + 1) | \Omega_k \Delta
    B_{r_k}(x_k) | \leq C m_k^{3/2},
  \end{align}
  for $C > 0$ depending only on $g$ and all $k$ large enough. Thus, we
  have that \eqref{eq:mchik} also holds with $\chi_{\Omega_k}$
  replaced with $\chi_{B_{r_k}(x_k)}$, and by Lipschitz continuity of
  $\bar g$ we obtain
  \begin{align}
    \label{eq:gxk}
    \bar g(x_k) = {1 \over m_k} \int_{\R^2} \chi_{B_{r_k}(x_k)}(x)
    \bar g(x_k) \, dx
    \leq  {1 \over m_k} \int_{\R^2} \chi_{B_{r_k}(x_k)}(x) \bar g(x) \, dx
    + C m_k^{1/2} \notag \\
    \leq {1 \over m_k} \int_{\R^2} \chi_{\Omega_k}(x) \bar g(x) \, dx
    + C' m_k^{1/2} \leq g(0) + C'' m_k^{1/2},
  \end{align}
  for some $C, C', C'' > 0$ depending only on $g$ and all $k$ large
  enough. In particular, $\bar g(x_k) = g(x_k)$ for all $k$
  sufficiently large, and by coercivity of $g$ the sequence of $x_k$
  is bounded. Thus, it is the desried sequence.

  Finally, to prove the third assertion of the theorem, we pass to the
  limit $k \to \infty$ in \eqref{eq:gxk}, after extracting a
  convergent subsequence, and use continuity of $g$.
\end{proof}

\section{The Euler--Lagrange equation: Proof of Theorem \ref{teocinque}}\label{sec:shapederivative}

The aim of this section is to obtain the Euler--Lagrange equation
satisfied by regular critical points of the functional $E_\lambda$. In
order to do this, we first compute the first variation of an auxiliary
functional which will be shown to be related to the capacitary energy.
	
Given an open set $\Omega\subset\R^2$, not necessarily bounded, 
    and a function
$f\in L^\frac 43({\R^2})$, we define
\begin{equation}\label{minJf}
 I_{\Om,f}(v) :=
 \begin{cases}
   \frac12\|v\|^2_{\mathring{H}^{\frac12}(\R^2)} -\int_{\R^2}fv\,dx &
   \text{if} \ v\in \mathring{H}^\frac12(\R^2) \ \text{and} \
   v|_{\Omega^c}=0, \\
   +\infty & \text{otherwise}.
 \end{cases}
\end{equation}
Notice that since the space $\mathring{H}^{\frac12}(\R^2)$
continuously embeds into $L^4(\R^2)$,
the functional $I_{\Om,f}$ admits a unique minimizer
$u_{\Omega,f}\in \mathring{H}^{\frac12}(\R^2)$, which satisfies
\begin{equation}\label{boiade}
\begin{cases}
  (-\Delta)^\frac12 u_{\Om,f}=f\qquad&\text{on }\Omega, \\
  u_{\Om,f}=0&\text{on }\Omega^c,
\end{cases}
\end{equation}
in the distributional sense, namely (see \cite[Eq. (4.14)]{LMM}):
  \begin{align}
    \int_\Omega u_{\Omega,f}  (-\Delta)^\frac12 \varphi \, dx =
    \int_\Omega f \varphi 
    \, dx \qquad \forall \varphi \in C^\infty_c(\Omega),
  \end{align}
  where
  \begin{align}
    \label{L12}
    (-\Delta)^\frac12  \varphi(x) := {1 \over 4 \pi} \int_{\R^2} {2
    \varphi(x) - \varphi(x - y) - 
    \varphi (x + y) \over |y|^3} \, dy \qquad x \in \R^2,
  \end{align}
  with the usual convention of extending $\varphi$ by zero outside
  $\Om$.  Furthermore, when
  $u_{\Om,f}|_\Om \in C^{1,\alpha}_\mathrm{loc}(\Omega) \cap
  L^\infty(\Om)$ for some $\alpha \in (0,1)$, we also have that
  \eqref{boiade} holds pointwise in $\Om$, with the definition of
  $(-\Delta)^\frac12$ in \eqref{L12} extended to such functions
  \cite[Section 3]{ro}. In addition, in this case we have
\begin{align}
  \label{Jfuf}
  J_f(\Omega):=\min I_{\Om,f} = -\frac 12 \int_\Omega u_{\Omega,f}
  f\,dx = -\frac 12 
  \int_\Omega u_{\Omega,f}\,  (-\Delta)^\frac12  u_{\Om,f}\,dx\,. 
\end{align}

 The following lemma gives a basic regularity result for the
  Dirichlet problem in \eqref{boiade}.

\begin{lemma}\label{uba}
  Let $f\in L^\infty(\R^2)\cap L^\frac 43(\R^2)$, let
  $\Omega\subset\R^2$ be an open set
  and let $u_{\Omega,f}$ be the minimizer of $I_{\Om,f}$.
  Then there exists a constant $C >0$ depending only on $f$ such that
  $\|u_{\Omega,f}\|_{L^\infty(\R^2)}\le C$.\\ If in
    addition $f|_\Om \in C^\alpha_\mathrm{loc}(\Om)$ for some
    $\alpha \in (0,1)$ then
    $u_{\Om,f} |_\Om \in C^{1,\alpha}_\mathrm{loc}(\Om)$.
\end{lemma}

\begin{proof}
  Let $\varphi\in \mathring{H}^\frac12(\R^2)$ be the unique solution
  to $(-\Delta)^\frac 12\varphi= -f$ in $\R^2$ (for a detailed
  discussion of the notion and the representations of the solution,
  see \cite[Section 4]{LMM}).  In particular, since by assumption
  $f \in L^{\frac43}(\R^2) \cap L^\infty(\R^2)$, from \cite[Lemma
  4.1]{LMM} it follows that $\varphi\in L^\infty(\R^2)$.  Furthermore,
  since we have
\begin{align}
  \frac 12\, \|v\|_{\mathring{H}^\frac12(\R^2)}^2 -\int_{\R^2} v f\,dx
  = \frac 12\, \|v+\varphi\|_{\mathring{H}^\frac12(\R^2)}^2 - \frac
  12\, \|\varphi\|_{\mathring{H}^\frac12(\R^2)}^2 \qquad \text{for any
    $v\in \mathring{H}^\frac12(\R^2)$,}
\end{align} 
we get that the function $w_{\Omega,f} := u_{\Omega,f}+\varphi$ 
solves the minimum problem
\begin{align}
  w_{\Om,f} = \mathrm{argmin} \left\{
    \|w\|^2_{\mathring{H}^{\frac12}(\R^2)} \ :\ w \in
    \mathring{H}^\frac12(\R^2),\, w|_{\Omega^c}=\varphi \right\}.
\end{align}

By an explicit computation, for any
  $w \in \mathring{H}^{\frac12}(\R^2)$ we have
  $\| \bar w \|_{\mathring{H}^{\frac12}(\R^2)} \leq \| w
  \|_{\mathring{H}^{\frac12}(\R^2)}$, where
  \begin{align}
    \bar w := \min( \max( w, -\| \varphi \|_{L^\infty(\R^2)}), \|
    \varphi \|_{L^\infty(\R^2)}).
  \end{align}
It then follows that
$ \|w_{\Omega,f}\|_{L^\infty(\R^2)}\le
\|\varphi\|_{L^\infty(\R^2)}$, yielding
\begin{align}
  \|u_{\Omega,f}\|_{L^\infty(\R^2)} \le
  \|w_{\Omega,f}\|_{L^\infty(\R^2)} + \|\varphi\|_{L^\infty(\R^2)}\le
  2\, \|\varphi\|_{L^\infty(\R^2)}.
\end{align}

Finally, H\"older regularity of the derivative of $u$ is an
  immediate consequence of \cite[Eq. (6.2)]{ro} (see also the references
  therein).
\end{proof}

We now recall the definition of the normal $\frac12$-derivative of the
function $u_{\Omega,f}$ vanishing at points of $\partial\Omega$:
\begin{equation}
  \partial^{1/2}_{\nu}u_{\Omega,f}(x):=\lim_{s\to0^+}\frac{u_{\Omega,f}(x-s\nu(x))
  }{s^{1/2}} \qquad \qquad x \in \partial \Omega,  
    \end{equation}
    where $\nu(x)$ is the outward unit normal vector.  We have the
    following result that will be crucial for the computation of
      the shape derivative of $J_f(\Om)$.

\begin{lemma}\label{uc}
  Let $\Omega_n\,,\Omega_\infty \subset \R^2$, $n\in\mathbb N$, be
  open sets whose boundaries are uniformly bounded and uniformly of
  class $C^{1,\alpha}$ for some $\alpha\in (0,1/2)$. Let
  $f\in L^\infty(\R^2)\cap L^\frac 43(\R^2)$ and assume that
  $\Omega_n\to\Omega_\infty$, as $n\to\infty$, in the Hausdorff
  distance.
  Then, for all $n\in\mathbb N\cup\{\infty\}$ the function
  $\partial^{1/2}_\nu u_{\Omega_n,f}$ can be continuously extended to
  a function $\overline D_n\in C^\alpha(\R^2)$ such that
  $\overline D_n\to \overline D_\infty$ as $n\to\infty$, locally
  uniformly in $\R^2$.
\end{lemma}

	\begin{proof}
          Denote $u_n := u_{\Omega_n,f}$ for simplicity.  Let
          $R_1>2R_0>0$ be such that
          $\partial \Omega_n\subset B_{R_0/2}(0)$ and
          $B_{R_0}(0)\subset B_{R_1/2}(x_0)$ for all
          $n\in\mathbb N\cup\{\infty\}$ and
          $x_0\in\partial\Omega_n$. Let also
          $\widetilde\Omega_n := \Omega_n\cap B_{R_0+R_1}(0)$.
          
          Notice that from Lemma \ref{uba} it follows that
          $\|u_n\|_{L^\infty(\R^2)}\le C$ for some constant $C>0$
          independent of $n$.  Then by \cite[Proposition 1.1]{ro-s},
          applied with $\Omega$ replaced by $\widetilde\Omega_n$ and
          $B_1(0)$ replaced by $B_{R_1}(x_0)$ for some
          $x_0\in\partial\Omega_n$, the sequence $(u_n)$ is uniformly
          bounded in $C^{1/2}({B_{R_0}(0)})$.  We observe that the
          $C^{1/2}$-estimate in \cite{ro-s} is uniform in $n$ since
          the involved constants depend only on the
          $C^{1,\alpha}$-norm of the boundary of
          $\partial \widetilde\Omega_n$.  As a consequence, by
          Arzel\`a-Ascoli Theorem, up to passing to a subsequence, the
          functions $u_n$ converge as $n\to\infty$ to $u^*$
          uniformly in $\overline B_{R_0}(0)$.

          To identify the limit function $u^*$, we establish the
          $\Gamma$-convergence of the functional $I_{\Om_n,f}$ to
          $I_{\Om_\infty,f}$ with respect to the weak convergence in
          $\mathring{H}^{\frac12}(\R^2)$. The latter is the natural
          topology, since the minimizers of $I_{\Om_n,f}$ are
          uniformly bounded in $\mathring{H}^{\frac12}(\R^2)$
          independently of $n$. 
          Indeed, by H\"older inequality we have 
          \begin{align}
            0 = I_{\Om_n,f} (0) \geq \frac12 \| u_n
            \|^2_{\mathring{H}^{\frac12}(\R^2)} - \| f
            \|_{L^{\frac43}(\R^2)} \| u_n \|_{L^4(\R^2)},
          \end{align}
          and the last term is dominated by the first term in the
          right-hand side by fractional Sobolev inequality
          \cite[Theorem 6.5]{DPV}. The $\Gamma-\liminf$ follows from
          lower-semicontinuity of the
          $\mathring{H}^{\frac12}(\R^2)$-norm and the continuity of
          the linear term, together with the fact that the limit
          function vanishes a.~e. in $\Omega_\infty^c$ by compact
          embedding of $\mathring{H}^{\frac12}(\R^2)$ into
          $L^p_\mathrm{loc}(\R^2)$ for any $p < 4$ \cite[Corollary
          7.2]{DPV}. Finally, the $\Gamma-\limsup$ follows by
          approximating the limit function by a function from
          $C^\infty_c(\Omega_\infty)$, for which we have pointwise
          convergence of $I_{\Om_n,f}$, and a diagonal argument. As a
          corollary to this result, we have that
          $u_n \rightharpoonup u_\infty$ in
          $\mathring{H}^{\frac12}(\R^2)$ and, hence, by uniqueness of
          the minimizer of $I_{\Om_\infty,f}$, we also have
            $u_n \to u_\infty$ a. e.  in $\R^2$. In particular,
          $u^* = u_\infty$ a. e. in $B_{R_0}(0)$.
          
          We now consider the functions $D_n:\widetilde\Omega_n\to\R$,
          $D_n(x) := u_n(x)/d_n^{1/2}(x)$, where
          $d_n(x) := \mathrm{dist}(x, \widetilde\Omega_n^c)$ and
          $n\in\mathbb N\cup\{\infty\}$.  Then by \cite[Theorem
          1.2]{ro-s} (see also \cite{DS}), applied as before with
          $\Omega$ replaced by $\widetilde\Omega_n$ and $B_1(0)$
          replaced by $B_{R_1}(x_0)$ for some
          $x_0\in\partial\Omega_n$, the sequence $(D_n)$ is uniformly
          bounded in $C^{\alpha}(\overline B_{R_0}(0))$.
          By classical extension theorems (see for instance
          \cite[Theorem 6.38]{gt}) for all $n\in\mathbb N$ we can
          extend $D_n$ to a function
          $\overline D_n:\, \overline B_{R_0}(0)\to\R$ such that
		\begin{align}
                  \|\overline D_n\|_{C^{\alpha}(\overline
                  B_{R_0}(0))}\le
                  C_0\|D_n\|_{C^\alpha(\overline B_{R_0}(0))}\le C, 
		\end{align}
		where the constants $C_0,\,C$ are independent of $n$.
                Again by Arzel\`a-Ascoli Theorem, up to passing to a
                subsequence, the functions $\overline D_n$ converge as
                $n\to\infty$ to a function
                $\overline D^* \in C^{\alpha}(\overline
                  B_{R_0}(0))$ uniformly. Moreover, from the
                convergence of $u_n$ to $u_\infty$ we get that
                $\overline D^*|_{\widetilde\Omega_\infty \cap
                  \overline B_{R_0}(0)}=D_\infty$.
		
          Finally, we observe that $\overline D_n$ is a
          continuous extension of $\partial_\nu^{1/2} u_n$ for all
          $n\in\mathbb N\cup\{\infty\}$, since we have, for any
          $x\in\partial\Omega_n$,
		\begin{align}
		\begin{aligned}
		\overline D_n(x)= \lim_{s\to 0^+}D_n(x-s\nu(x))
		= \lim_{s\to 0^+}\frac{u_n(x-s\nu(x))}{d_n(x-s\nu_{\Omega_n}(x))^{1/2}}
		=\partial_\nu^{1/2} u_n(x),
		\end{aligned}
		\end{align}
                concluding the proof.
	\end{proof}
	
	\begin{corollary}\label{1luglio}
          Under the assumptions of Lemma \ref{uc}, let
          $x_n\in\partial\Omega_n$ and $x\in\partial\Omega_\infty$ be
          such that $x_n\to x\in\partial \Omega_\infty$. Then
          $\partial^{1/2}_\nu u_n(x_n)\to \partial^{1/2}_\nu
          u_\infty(x)$ as $n\to+\infty$.
	\end{corollary}
	
	\begin{proof}
          Consider the extensions $\overline D_n$,
          $n\in\mathbb N\cup\{\infty\}$, constructed in the proof of
          the previous lemma. Then we have
		\begin{align}
		\begin{aligned}
                  |\partial^{1/2}_\nu u_n(x_n)-\partial^{1/2}_\nu
                  u_\infty(x)|&= |\overline D_n(x_n)-\overline
                  D_\infty(x)|\\ 
                  &\le |\overline D_n(x_n)-\overline
                  D_n(x)|+|\overline D_n(x)-\overline D_\infty(x)|
		\end{aligned}
		\end{align}
		and the right-hand side of the latter inequality
                converges to $0$ as $n\to+\infty$.
	\end{proof}

        We now compute the first variation of the functional $J_f$.
        We note that for bounded domains and under stronger regularity
        assumptions such a computation was carried out in
        \cite{DGV13}, with a relatively long and technical proof. Here
        we provide an alternative, shorter proof, that also covers the
        case of unbounded domains and weaker assumptions on the
        regularity of $f$ and $\partial \Om$.

\begin{theorem}\label{dgv}
  Let $f\in L^\infty(\R^2)\cap L^\frac 43(\R^2)$ be such that
  $f|_\Om \in C^\alpha_\mathrm{loc}(\Om)$ for some $\alpha \in
  (0,1)$. Let $\Omega$ be an open set with compact boundary of class
  $C^2$, and let $u_{\Omega,f}$ be the unique minimizer of
  $I_{\Omega,f}$.  Let $\zeta \in C^\infty(\R^2, \R^2)$, and let
  $(\Phi_t)_{t\in \R}$ be a smooth family of diffeomorphisms of the
  plane satisfying $\Phi_0 = \mathrm{Id}$ and
  $\left. {d \over dt} \Phi_t \right|_{t = 0} = \zeta$.  Then, if
  $\nu$ is the outward pointing normal vector to $\partial\Omega$, the
  normal $\frac12$-derivative $\partial^{1/2}_\nu u_{\Om,f}$ is
  well-defined and belongs to $C^\beta(\partial\Omega)$ for any
  $\beta\in (0,1/2)$.  Moreover, we have
\begin{equation}
  \left. \frac{d}{dt}J_f(\Phi_t(\Omega))\right|_{t=0}=-\frac\pi 8\,
  \int_{\partial\Omega}(\partial^{1/2}_\nu
  u_{\Omega,f}(x))^2\zeta(x)\cdot \nu(x)\,d\mathcal H^{1}(x).
\end{equation}
\end{theorem}

\begin{proof}
  Let $\Om_t := \Phi_t(\Om)$. Since $\partial\Omega$ is of class
  $C^2$, for all $x \in \partial \Omega_t$ and $|t|$ small enough we
  can write
\begin{align}
  \label{Phim1tx}
\Phi_t^{-1}(x)=x+t \rho_t(x)\nu_t(x),
\end{align}
where $\rho_t\in C^{2}(\partial\Omega_t)$ is a scalar function and
$\nu_t$ is the unit outward normal to $\partial
\Omega_t$. Furthermore, the right-hand side of \eqref{Phim1tx}
establishes a bijection between $\partial \Om_t$ and $\partial \Om$,
and we have
\begin{align}
  \label{rho0}
  \rho_0(x) := \lim_{t \to 0} \rho_t(x) = -\zeta(x) \cdot
  \nu(x) \qquad \forall x \in \partial 
  \Omega.
\end{align}

For $t > 0$ sufficiently small, let $\Om_t \subset \Om$ be a regular
inward deformation of $\Om$, namely, $\Om_t$ is such that
\eqref{Phim1tx} holds true with some $\rho_t \geq 0$.  Note that it is
enough to consider inward perturbations, since for outward
perturbations one would simply interchange the roles of $\Om_t$ and
$\Om$ in the argument below.

We denote $u:=u_{\Omega,f}$ and $u_t:=u_{\Omega_t,f}$ for simplicity.
Recall that $u$ and $\,u_t$ solve pointwise
\begin{equation}\label{eq:3}
  \begin{cases}  
    (-\Delta)^{\frac12} u = f \quad &\text{in} \quad \Omega, \\
    u = 0 & \text{in} \quad \Omega^c,
  \end{cases}
\end{equation}
and
\begin{equation}
  \label{eq:5}
  \begin{cases}
    (-\Delta)^{\frac12} u_t = f \quad &\text{in} \quad \Omega_t, \\
    u_t = 0 & \text{in} \quad \Omega_t^c.
  \end{cases}
\end{equation}
In particular, by \cite[Theorem 1.2]{ro-s} we have
$|u(x)| \leq C \sqrt{\text{dist}(x, \partial \Omega)}$ for some
constant $C>0$, which in turn implies that
$|(-\Delta)^{\frac12} u(x)|\le C/\sqrt{\text{dist}(x, \partial
  \Omega)}$, and the same holds for the function $u_t$, with $\Om$
replaced by $\Om_t$ and the constant $C$ independent of $t$ for all
small enough $t$. These estimates justify all the computations of
integrals involving $u$ and $u_t$ below.
      
From \eqref{eq:3}, \eqref{eq:5} and \eqref{Jfuf} we get
\begin{align}
  \label{eq:6}
  J_f(\Om_t) - J_f(\Om)
  & 
    = \frac12 \int_{\R^2} u  (-\Delta)^{\frac12} u \, dx -
    \frac12 \int_{\R^2} u_t  (-\Delta)^{\frac12} u_t \, dx \notag \\
  & = -\frac12 \int_{\R^2} (u_t + u)
    (-\Delta)^{\frac12} (u_t - u) \, dx \\
  & = -\frac12 \int_{\Om \setminus \Om_t} u
    (-\Delta)^{\frac12} u_t \, dx + \frac12 \int_{\Om
    \setminus \Om_t} u f\, dx. \notag 
\end{align}
The last term in \eqref{eq:6} satisfies
\begin{align}
  \left| \frac12\int_{\Om\setminus\Om_t}u f \,dx \right|
  \le C\, \|f\|_{L^\infty(\Om\setminus\Om_t)}\, 
  \|{\rho_t}\|^\frac 12_{L^\infty(\partial\Om)}\,
  t^\frac 12\, |\Om\setminus\Om_t|
  \le C' t^{\frac32}=o(t).
\end{align}
We thus focus on the first term of the right-hand side of
\eqref{eq:6}. We have:
\begin{align}
  \label{eq:7}
  -\frac12 \int_{\Om \setminus \Om_t} u
  (-\Delta)^{\frac12} u_t \, dx 
  & = -{1 \over 8 \pi}\,
    \int_{\Om \setminus \Om_t} \int_{\R^2}
    u(x) \, {2 u_t(x) - u_t(x - y) - u_t(x + y) \over |y|^3}
    \, dy \, dx \notag 
  \\ 
  & =\ {1 \over 4 \pi} \int_{\Om \setminus \Om_t} \int_{\Om_t} {u(x)
    u_t(y) \over |x - y|^3} \, dy \, dx. 
\end{align}
	
Next we split the integral over $\Om_t$ in \eqref{eq:7} into
integrals over $\Om^R$ and $\Om_t \backslash \Om^R$, where
\begin{equation}\label{eq:16}
  \Om^R := \{ x \in \Om : \text{dist}(x, \Om^c) > R\}
\end{equation}
and $R > 0$ is such that $\partial \Om^R$ is of class $C^2$ and
$\Om_t \backslash \Om^R$ consists of a union of disjoint strip-like
domains. We have
\begin{align}
	\label{eq:19}
          & -\frac12 \int_{\Om \setminus \Om_t} u
            (-\Delta)^{\frac12} u_t \, dx \notag \\
          &  =  {1 \over 4 \pi} \int_{\Om \setminus \Om_t} \int_{\Om^R} {u(x)
            u_t(y) \over |x - y|^3} \, dy \, dx +  {1 \over 4 \pi} \int_{\Om
            \setminus \Om_t} \int_{\Om_t \setminus \Om^R} {u(x)
            u_t(y) \over |x - y|^3} \, dy \, dx \notag \\
          & =  {1 \over 4 \pi} \int_{\Om \setminus \Om_t} \int_{\Om^R} {u(x)
            u_t(y) \over |x - y|^3} \, dy \, dx +  {1 \over 4 \pi} \notag \\
          & \times \int_{\partial \Om_t}
            \int_{\partial \Om_t} \int_0^{t {\rho_t}(x)} \int_0^R  {u(x +
            s \nu_t(x)) u_t(y - s' \nu_t(y)) \over |x - y +
            s \nu_t(x) + s' \nu_t(y)|^3} \notag \\
          & \qquad \qquad \times (1 + s \kappa(x)) (1 - s'
            \kappa(y)) \, ds' \, ds 
            \, d\mathcal H^1(y) \, d \mathcal H^1(x),
	\end{align}
	where $\kappa$ is the curvature of $\partial \Om_t$, positive
        if $\Om_t$ is convex.  As above, one can check that the first
        term on the right-hand side of \eqref{eq:19} is
        $O(t^{3/2} R^{-2})$ for all $t$ small enough, hence we can
        focus again just on the second term.
	
	To estimate the last integral in the right-hand side of
        \eqref{eq:19}, we first observe that the curvature
        contributions inside the brackets can be bounded by $O(R)$
        and, therefore, a posteriori give rise to errors of order
          $O(Rt)$ for all $t$ small enough, as the integral itself
          will be shown to be $O(t)$. Thus, we have
	\begin{multline}
	\label{eq:10}
          -\frac12 \int_{\Om \setminus \Om_t} u
            (-\Delta)^{\frac12} u_t \, dx = \frac{1}{4 \pi} 
           \\          
          \times \int_{\partial \Om_t}
            \int_{\partial \Om_t} \int_0^{t {\rho_t}(x)} \int_0^{R}  {u(x +
            s \nu_t(x)) u_t(y - s' \nu_t(y)) \over |x - y +
            s \nu_t(x) + s' \nu_t(y)|^3} \, ds' \, ds \, d\mathcal H^1(y) \, d
            \mathcal H^1(x)  \\
          + O(t^{3/2} R^{-2} ) +O(Rt).
	\end{multline}
	
For $x \in \partial \Om_t$, we let
\begin{equation}\label{J0}
  F(x):=\int_{\partial \Om_t}\int_0^{t{\rho_t}(x)}\int_0^R
  \frac{u(x+s\nu_t(x))u_t(y-s'\nu_t(y))}{|x-y+s\nu_t(x)
    +s'\nu_t(y)|^3}\,ds'\,ds\,d\mathcal H^1(y) ,
\end{equation}
and split the integral over $y$ into a near-field part
\begin{align}
  \label{HR}
  F_R(x):=\int_{\partial \Om_t \cap
  B_R(x)}\int_0^{t{\rho_t}(x)}\int_0^R 
  \frac{u(x+s\nu_t(x))u_t(y-s'\nu_t(y))}{|x-y+s\nu_t(x)
  +s'\nu_t(y)|^3}\,ds'\,ds\,d\mathcal  
  H^1(y) ,
\end{align}
and the far field part $F(x) - F_R(x)$. As with \eqref{eq:10}, the
latter may be estimated to be $O(t^{3/2} R^{-2})$, so we focus on the
computation of $F_R(x)$. To that end, we let $y = y(\sigma)$ be the
arc-length parametrization of $\partial \Omega_t \cap B_R(x)$ relative
to $x$ and observe that
\begin{itemize}
\item[(i)]
  $(y(\sigma)-x)\cdot\nu_t(x)=O(\sigma^2)$,
\item[(ii)]$\nu_t(x)\cdot\nu_t(y(\sigma))=1 +
  O(\sigma^2)$,
\item[(iii)] $|y(\sigma)-x|=\sigma+O(\sigma^3)$,
\end{itemize}
uniformly in $x$ and $t$.  Therefore,
\begin{align}
  |y(\sigma)-x-
  &s\nu_t(x)-s'\nu_t(y(\sigma))|^2=|y(\sigma)-x-(s+s')\nu_t(x)-s'(\nu_t(y(\sigma)-\nu_t(x)))|^2
    \notag \\ 
  &=|y(\sigma)-x|^2+(s+s')^2-2(y(\sigma)-x)\cdot
    \nu_t(x)
    (s+s')+|s'|^2|\nu_t(y(\sigma))-\nu_t(x)|^2\notag \\
  &\quad-2s'(y(\sigma)-x)\cdot(\nu_t(y(\sigma))-\nu_t(x))+2(s+s')s'\nu_t(x)\cdot(\nu_t(y(\sigma))-\nu_t(x))
    \notag
  \\
  &=\sigma^2+(s+s')^2+O(\sigma^4)+O(\sigma^2R) + O(\sigma^2 R^2) \notag \\
  &=(\sigma^2+(s+s')^2)(1+O(R)),
\end{align}
again, uniformly in $x$ and $t$, for all $R$ small enough.  Thus we
have
\begin{equation}\label{J1}
  |y(\sigma)-x-s\nu_t(x)-s'\nu_t(y(\sigma))|^{-3}=(\sigma^2+(s+s')^2
  )^{-\frac32}(1+O(R)).  
\end{equation}

By the uniform convergence of $\partial^{1/2}_\nu u_t$ to
$\partial^{1/2}_\nu u$ as $t\to0$, and the fact that
$\partial^{1/2}_\nu u_t$ is of class $C^{\beta}(\partial\Omega_t)$ for
all $\beta\in (0,1/2)$ (by Lemma \ref{uc}), we have that
\begin{equation}\label{J2}
\begin{aligned}
  u(x+s\nu_t(x))&= (1+o_t(1)) \, \partial^{1/2}_\nu
  u(x+t{\rho_t}(x)\nu_t(x))\sqrt{t{\rho_t}(x)-s}\\
  &= (1+o_t(1)) \, \partial^{1/2}_\nu
  u_t(x)\sqrt{t{\rho_t}(x)-s}
\end{aligned}
\end{equation}
 and
 \begin{equation}\label{J3}
 \begin{aligned}
   u_t(y(\sigma)-s'\nu_t(y(\sigma))&= (1+o_t(1)) \,
   \partial^{1/2}_\nu u_t(y(\sigma))\sqrt{s'}\\
   & = (1+o_t(1) + o_R(1)) \, \partial^{1/2}_\nu u_t(x)\sqrt{s'}.
  \end{aligned}
 \end{equation}
 Plugging \eqref{J1},\eqref{J2} and \eqref{J3} into \eqref{J0}, we get
\begin{equation}\label{J4}
  F_R(x) 
  =(1+o_t(1) + o_R(1))
  \int_{\sigma_R^-(x)}^{\sigma_R^+(x)}\int_0^{t \rho_t(x)}\int_0^R
  \frac{|\partial^{1/2}_\nu u_t(x)|^2\sqrt{(t{\rho_t}(x)-s)s'}}{(\sigma^2+(s+s')^2)^{3/2}}\,
  ds'\,ds\,d\sigma,
\end{equation}
where $\sigma_R^\pm(x) = \pm R + O(R^3)$.

Observe that $F_R(x) = 0$ if $\rho_t(x) = 0$. If $\rho_t(x) > 0$, we
can perform the change of variables
\begin{align}
  z={s \over t{\rho_t}(x)}, \qquad 
  z'={s' \over t{\rho_t}(x)}, \qquad \zeta={\sigma \over
  t{\rho_t}(x)},    
\end{align}
to obtain
\begin{align}
  F_R(x)=(1
  & +o_t(1) +o_R(1)) \,
    t{\rho_t}(x)|\partial^{1/2}_\nu
    u_t(x)|^2 \notag \\
  & \times \int_{\sigma_R^-(x) / (t \rho_t(x))}^{\sigma_R^+(x) / (t
    \rho_t(x))} \int_0^1
    \int_0^{R/(t{\rho_t}(x))}
    \frac{\sqrt{(1-z)z'}}{(\zeta^2+(z+z')^2)^{3/2}} \, dz' \, dz \, 
    d \zeta,    
\end{align}
which is also valid if $\rho_t(x) = 0$.  By Dominated Convergence
Theorem, as $t \to 0$ the integral in the right-hand side converges to
\begin{align}
  \int_{-\infty}^\infty \int_0^1
  \int_0^\infty
  \frac{\sqrt{(1-z)z'}}{(\zeta^2+(z+z')^2)^{3/2}} \, dz' \, dz \, 
  d \zeta = 2 \int_0^1 \int_0^\infty
  \frac{\sqrt{(1-z)z'}}{(z+z')^2} \, dz' \, dz = {\pi^2 \over 2}.
\end{align}	
We thus have
	\begin{align}
	\label{eq:11}
          -\frac12 \int_{\Om \setminus \Om_t} u
          (-\Delta)^{\frac12} u_t \, dx \notag 
          = (1+o_t(1) +o_R(1)) \frac{\pi t}{8} \int_{\partial
          \Om_t } |\partial_\nu  
          u_t(x)|^2 \rho_t(x) \, d
          \mathcal H^1(x),
	\end{align}
        so that by the above estimates and the uniform continuity of
          $\rho_t$ and $\partial^{1/2}_\nu u_t$ in $t$, we get
	\begin{equation}\label{eq:12}
	\begin{aligned}
          \lim_{t \to 0} {J_f(\Om_t) - J_f(\Om) \over t} &=
          (1+o_R(1)) \, \frac{\pi}{8} \int_{\partial \Om}
          |\partial_\nu u(x)|^2 \rho_0(x) d \mathcal H^1(x).
	\end{aligned}
	\end{equation} 
        Finally, by \eqref{rho0} the thesis follows by sending
        $R \to 0$ in \eqref{eq:12} and Lemma \ref{uc}.
\end{proof}

\begin{remark}\rm 
  As was mentioned earlier, for bounded domains and under stronger
  regularity assumptions the result in Theorem \ref{dgv} was obtained
  in \cite{DGV13} with a different proof. In \cite[Theorem 1]{DGV13},
  the first variation is stated with a non-explicit constant, but an
  analysis of the proof shows that their constant agrees with ours, as
  it should.  Our proof exploits the boundary regularity for non-local
  elliptic problems developed in \cite{ro-s} (see also the survey
  \cite{ro} and \cite{DS}), which simplifies the proof even for
  bounded domains.
\end{remark}

We are now in a position to compute the first variation of the
functional $\I_1$ on $C^2-$regular bounded sets, and consequently the
Euler--Lagrange equation for $E_{\lambda}$ for such sets.  Recalling
\eqref{capuno}, it is enough to compute the first variation of the
$\frac 12$-capacity ${\rm cap}_1(\Om)$, which follows directly from
Theorem \ref{dgv}, as we show below.

\begin{theorem}\label{teocinquebis}
  Let $\Omega$ be a compact set with boundary of class $C^2$, let
  $\nu$ be the outward pointing normal vector to $\partial\Omega$
  and let $u_\Om$ be the $\frac12$-capacitary potential of $\Omega$
    defined in \eqref{uOm}. Then, the $1/2$-derivative
  $\partial^{1/2}_\nu u_\Om$ is well-defined and belongs to
  $C^\beta(\partial\Omega)$ for any $\beta\in (0,1/2)$.  Moreover,
  letting $\zeta$ and $\Phi_t$ be as in Theorem \ref{dgv}, there holds
	\begin{equation}
          \left. \frac{d}{dt}{\rm cap}_1(\Phi_t(\Omega)) \right|_{t=0}= 
          \frac{\pi}{4}\,\int_{\partial\Omega}(\partial^{1/2}_\nu
          u_\Omega(x))^2\zeta(x)\cdot \nu(x)\,d\mathcal H^{1}(x). 
	\end{equation}
\end{theorem}

\begin{proof}
  Let $\varphi\in C^\infty_c(\R^2)$ be such that $\varphi=1$ in an
  open neighborhood of $\Omega$.  Observe that our choice of $\varphi$
  implies that, if
\begin{align}
f(x):=-(-\Delta)^\frac12\varphi (x)= -\frac{1}{4\pi} \, \int_{\R^2} \frac{2\varphi(x)-\varphi(x-y)-\varphi(x+y)}{|y|^3}\,dy,
\end{align}
then $\in L^\infty(\R^2)\cap L^\frac{4}{3}(\R^2) \cap C^{0,1}(\R^2)$.
Notice also that any test function $u$ in the definition of
${\rm cap}_1(\Om)$ such that $u=1$ on $\Om$ can be put in
correspondence with a test function $v=u-\varphi$ in the definition of
the auxiliary functional $I_{\Om^c,f}$ in \eqref{minJf}. Moreover,
since
\begin{align}
  \frac 12\, \| u\|_{\mathring{H}^\frac12(\R^2)}^2=\frac 12\,
  \| v\|_{\mathring{H}^\frac12(\R^2)}^2+\frac 12\,  
  \|\varphi\|_{\mathring{H}^\frac12(\R^2)}^2-\int_{\R^2} v f \, dx,
\end{align} 
we get that
\begin{align}
\frac 12\,{\rm cap}_1(\Omega)=\frac 12\, \|\varphi\|_{\mathring{H}^\frac12(\R^2)}^2+J_f(\Omega^c),
\end{align}
and the minimizer $u_\Omega$ satisfies
$u_\Omega=v_{\Omega^c,f} + \varphi$, where $v_{\Om^c,f}$ is
  the minimizer of $I_{\Om^c,f}$. Observing also that
$ \partial^{1/2}_\nu u_\Omega = \partial^{1/2}_\nu
v_{\Omega^c,f}$, the conclusion follows by Theorem \ref{dgv}.
\end{proof}

\smallskip
\noindent
Finally, Theorem \ref{teocinque} is a direct consequence of Theorem
\ref{teocinquebis}, together with \eqref{capuno} and \eqref{uOm}.

\bibliographystyle{plain}

\end{document}